\documentclass[12pt]{amsart}

\usepackage[utf8]{inputenc} 

\usepackage{epsfig}
\usepackage{indentfirst, latexsym, amssymb, enumerate, amsmath, graphicx}
\usepackage{colortbl}
\usepackage{subcaption}
\usepackage{mathtools}
\usepackage{mathrsfs}
\usepackage{amsfonts}
\usepackage{pgfplots}

\usepackage{lmodern}
\usepackage{amsthm}
\usepackage[T1]{fontenc}
\usepackage[english,activeacute]{babel}
\usepackage{mathtools}
\usepackage{amssymb}

\usepackage{tikz}
\usepackage{graphicx}
\graphicspath{ {C:\Users\HP\Desktop\Master\Doctorado\Articulo1} }
\usepackage{float}
\usepackage[rightcaption]{sidecap}
\usepackage{subfig} 
\usepackage[
a4paper,
textwidth=80ex,
]{geometry}

\newtheorem{definition}{Definition}
\newtheorem{prop}[definition]{Propositon}
\newtheorem{theorem}[definition]{Theorem}
\newtheorem{lemma}[definition]{Lemma}
\newtheorem{remark}[definition]{Remark}

	\title[Traveling Waves at the Interface of Cell Populations]{Biomechanical Effects on Traveling Waves at the Interface of Cell Populations}
	
\author[J.~Campos]{Juan Campos}
		
		\author[C.~Pulido]{Carlos Pulido}
		
\author[J.~Soler]{Juan Soler}
\address{\sc C. Pulido, Juan Campos \& Juan Soler, Departamento de Matem\'atica Aplicada and Research Unit ``Modeling Nature'' (MNat), Facultad de Ciencias, Universidad de Granada, 18071 Granada, Spain}
\email{campos@ugr.es, cpulidog@ugr.es, jsoler@ugr.es}
\begin{document}



\thanks{\textbf{Acknowledgment.} We would like to thank Jean-François Joanny for his stimulating discussions on this work during the BIOMAT-23 school. This work has been partially supported by the State Research Agency of the Spanish Ministry of Science and FEDER-EU, project PID2022-137228OB-I00 (MICIU/AEI /10.13039/501100011033); by Modeling Nature Research Unit, Grant QUAL21-011 funded by Consejer\'ia de Universidad, Investigaci\'on e Innovaci\'on (Junta de Andaluc\'ia).}
				
		\begin{abstract}
		This study builds upon a model proposed by Joanny and collaborators that examines the dynamics of interfaces between two distinct cell populations, particularly during tumor growth in healthy tissues. This framework leads to the investigation of a general model with a non-local and strongly nonlinear advection term representing the biomechanical interaction between both populations. The model captures the evolution of front propagation, reflecting the interaction between cell population dynamics and tissue mechanics. We explore the existence of traveling wave solutions to this problem and establish upper and lower bounds on the propagation speed across various biological parameters. In this way, the model accounts for both biomechanical and biochemical interactions.			
		\end{abstract}

	\maketitle
	
	\section{Introduction:}
	
The aim of this paper is to study certain patterns, particularly traveling waves, underlying the following nonlinear and nonlocal model that arises in the interface of cellular populations:
\begin{equation}\label{model}
	\begin{split}
		\partial_T \phi+V \partial_X \phi \, &= \, \partial_X^2 \phi+\phi(1-\phi), \\
		\Lambda^2 \partial_X^2 V-V \, &= \, 2 \Lambda V_0 \partial_X \left(\phi+\beta (\partial_X \phi)^2\right),
	\end{split}
\end{equation}
where $\phi$ represents the interface and $V$ is the average propagation velocity of the system. The model was proposed by J-F Joanny {\it et al.} \cite{Joanny}, and it describes the evolution of interfaces between two distinct cell populations in the presence of cell division and cell death. The focus is on the effects of mechanical coupling between the populations and its impact on cell growth and pressure. Additionally, the model considers how the surrounding environment exerts pressure on the system, thereby influencing the growth dynamics of the populations.
	
The study of biomechanical properties, particularly the influence of stress and pressure on cellular behavior, presents a fascinating challenge, especially in biological tissues composed of various cell populations. This research aims to understand how pressure affects the behavior of the separation interface between these populations, which continuously interacts with biochemical effects. Numerous experiments have demonstrated that pressure exerted on a cell population significantly impacts its dynamics, particularly its growth \cite{Cheng,Guevorkian,Helmlinger}. Consequently, several models have recently been developed to investigate how these populations grow as a function of internal pressures within a population, as well as the pressures exerted by different populations on each other and by the surrounding environment \cite{PerthameQuiros,PerthameNadin,Perthame2,Joanny,Shraiman}.

The equation governing the movement of the interface arises from the hypothesis that the evolution of the two populations resembles the behavior of two fluids, whose mutual pressure connects their dynamics \cite{Joanny}. Studying the interface between these two populations enables us to predict their interactions and potential invasion patterns. In this context, the analysis and characterization of possible traveling waves, i.e., solutions of the form $\phi(x-\sigma t)$, $V(x-\sigma t)$, as a function of the biological parameters of the system:   $V_0$, $\beta$, $\Lambda$, and the wave speed $\sigma$, will shed light on the dynamics of the interface and its consequences. From a mathematical perspective, this analysis will require of a combination of techniques from partial differential equations, dynamical systems, and degree theory, depending on the various constants of the system.

Let us briefly outline the derivation of the model. The initial idea is to consider that the dynamics of each cell population are governed by the evolution equation:
\begin{equation}
	\partial_t n_i + \partial_x (n_i v_x^i)=n_ik(P_i^h-P),\quad i=1,2;
	\label{evol}
\end{equation}
where $n_i$ is the density of population $i$, $v_i$ is the velocity field of the population $i$, and $k$ is the ratio of celular division of the population $i$. The hypothesis that allows for the development of the model is that the cell growth of each population will depend on the pressure exerted on it, through a pressure threshold (Homeostatic pressure) that the cell can withstand. Thus, if the pressures are higher than this threshold, cell death is favored and if they are lower, growth is favored.

We assume that the system formed by both populations satisfies the hypothesis of incompressibility:
$$n_1\Omega_1+n_2\Omega_2=1 ,$$
where $\Omega_i$ is the constant volumen of each cell $i$. Setting $\phi=n_1\Omega_1$ and defining the average velocity as:
$$v=v_1\phi+v_2(1-\phi) ,$$
we can deduce, under certain considerations, that $\phi$ satisfies the equation:
\begin{equation}
	\partial_t\phi+v_x\partial_x\phi=D\partial^2_x \phi+\phi(1-\phi)\kappa (P^h_1-P^h_2).
	\label{phi}
\end{equation}

On the other hand, it remains to determine the equation satisfied by the average velocity. To achieve this, the system is considered as a viscous fluid. From its deformation tensor (\(\sigma_{\alpha\beta}\)) and considering that the fluid moves through the medium with friction (\(\partial_\alpha \sigma_{\alpha\beta} = -\gamma v\)), the following equation is obtained:
\begin{equation}
	-\Delta P^h\partial_x \phi+\partial_x^2v_x\left( \frac{1}{\kappa}+\frac{4}{3}\eta\right) -\frac{4}{3}B\partial_x^2\phi\partial_x\phi=\gamma v_x.
	\label{v}
\end{equation}

By adimensionalizing the equation, we obtain the  dimensionless form \eqref{model} for further study.
The equation for the average propagation velocity \( V \) follows a Helmholtz equation. However, in contrast to other models  for fluid interfaces, where the propagation velocity is typically related to the pressure through Darcy's law or the Brinkman law (\cite{PerthameQuiros},\cite{Perthame2},\cite{PerthameDebiez}), our model exhibits a more intricate relationship, involving the gradient of the density and the variation of its modulus. In that other cases, we have 
\[ 
\begin{split}
-\frac{\mu}{\kappa}V &=\partial_XP\; \mbox{ (for Darcy's law),}\\
\nu\partial^2_{XX} V-\frac{\mu}{\kappa}V &=\partial_XP\; \mbox{ (for Brinkman's law),}
\end{split}
\]
where, in these models, the pressure is expressed as a power of the population densities, i.e.  $\partial_XP= \phi^\gamma$.

As we mentioned earlier,  the aim of the paper is to analyze the existence of traveling waves solutions of \eqref{model}.  One of the main challenges of the model is how to manage the non-local advection term that arises, given that the velocity can be expressed as:
\begin{equation}
	V=\Gamma*\left(2 \Lambda V_0 \partial_X \left(\phi+\beta (\partial_X \phi)^2\right)\right),
	\label{Voriginal}
\end{equation}
where $\Gamma$ is the kernel associated with the Helmholtz operator $\Lambda^2\partial^2_{XX}V-V=\delta$. The objective is to determine how the different parameters of the system influence the existence or absence of traveling waves. By analyzing the role of these parameters, we aim to understand the conditions under which traveling waves can form and propagate, as well as the  characteristics of these waves. This analysis will provide insights into the dynamic behavior of the system and its response to various internal and external factors.

The existence of this type of solutions for problems in which nonlocal terms appear, either in the reaction term or in the advection term, have been analyzed by different authors, and always considering that the nonlocality had the form $K*\phi$, where K could be a kernel in some $L^p$ space (\cite{HamelHenderson},\cite{BerestyckiPerthame},\cite{Henderson}), or the Helmholtz kernel (\cite{Chmaj},\cite{PerthameDebiez}), which in our case corresponds to \eqref{Voriginal}, or some sort of non local diffusion kernel (\cite{kuehn},\cite{Chmaj},\cite{Roquejoffre}), like in the fractional Laplacian case. 

Seeking traveling wave profiles \(\phi(t,x) = \phi(x-\sigma t)\) that solve the partial differential equation \eqref{model}, we derive the following system of second-order differential equations:
\begin{equation}
	\begin{split}
		-\sigma\phi^\prime+V\phi^\prime&=\phi^{\prime\prime}+\phi(1-\phi), \\
		\Lambda^2 V^{\prime\prime}-V&=2\Lambda V_0\phi^\prime(1+\beta\phi^{\prime\prime}),
	\end{split}
	\label{prob1}
\end{equation}
where  $\sigma>0$ is  the wave speed. The solutions to this differential equation represent the stationary profiles of traveling waves moving at a velocity \(\sigma\). Our goal is to find decreasing profiles that satisfy the boundary conditions \(\phi(-\infty) = 1\) and \(\phi(+\infty) = 0\).

Redefining the parameters as \( a = 2\Lambda V_0 \) and \( b = 2\Lambda V_0\beta \), the boundary problem to be studied then takes the form:
\begin{equation}
	\begin{split}
		-\sigma\phi^\prime+V\phi^\prime&=\phi^{\prime\prime}+\phi(1-\phi), \\
		\Lambda^2 V^{\prime\prime}-V&=a\phi^\prime+b\phi^\prime\phi^{\prime\prime},\\
		\phi(-\infty)=&1, \quad \phi(+\infty)=0.
	\end{split}
	\label{modeloriginal}
\end{equation} 
The existence of such traveling wave patterns will depend on the different values taken by the parameters \(a\), \(b\), and \(\Lambda\), as well as their relationship with the wave speed \(\sigma\).

In this paper, we will address the study of the existence of traveling waves using two different techniques.

Firstly, we focus on the case \(\Lambda = 0\) in \eqref{modeloriginal}, where the system exhibits local, nonlinear velocity behavior. By applying techniques from dynamical systems theory, we demonstrate the existence of solutions under specific parametric conditions, as detailed in the following result.

\begin{theorem} \label{teorema1}
	There exists \(\sigma^* = \sigma^*(a,b)\) such that there is a unique  solution, up to translations, to the problem \eqref{probi2}, for every \(\sigma \geq \sigma^*(a,b)\). Furthermore,
	\begin{itemize}
		\item[\rm (T1)] $\sigma^*\geq2$, and if $\max\{a,b\}\leq2$, then $\sigma^*=2$.
		\item[\rm (T2)] If $b=0$ and $a>2$, then
		$$
		\left\{\begin{array}{ll}
			2 & \text { if }   a\leq 3+2\sqrt{2} , \\
			\frac{a-1}{\sqrt{a}} & \text { if } a>3+2\sqrt{2}, 
		\end{array}\right.\leq\sigma^*(a,0)\leq \left\{\begin{array}{ll}
			\sqrt{\frac{a^2+4}{a}} & \text { if } 2 \leq a < a^*, \\
			2+\frac{a}{8} & \text { if } a^* \leq a < 16, \\
			\sqrt{a} & \text { if } 16\leq a,
		\end{array}\right.
		$$
		where $a^*$ is the unique root of $a^3-32a^2+256a-256=0$ in the interval $]2,16[$.
		\item[\rm (T3)] If $b>0$,
		$$
		2\leq\sigma^*\leq \sqrt{\frac{\sqrt{a^{2}+8 a+4 b+16}+a+4}{2}}.
		$$
		Moreover, if $a^2\geq4b$,
		$$
		\sigma^*\geq\max\left\lbrace 2,\frac{2b-(a-\sqrt{a^2-4b})}{\sqrt{2b\left( a-\sqrt{a^2-4b}\right) }}\right\rbrace.
		$$
	\end{itemize}
\end{theorem}

\begin{figure}[h!]
	{\includegraphics[width=1.85in]{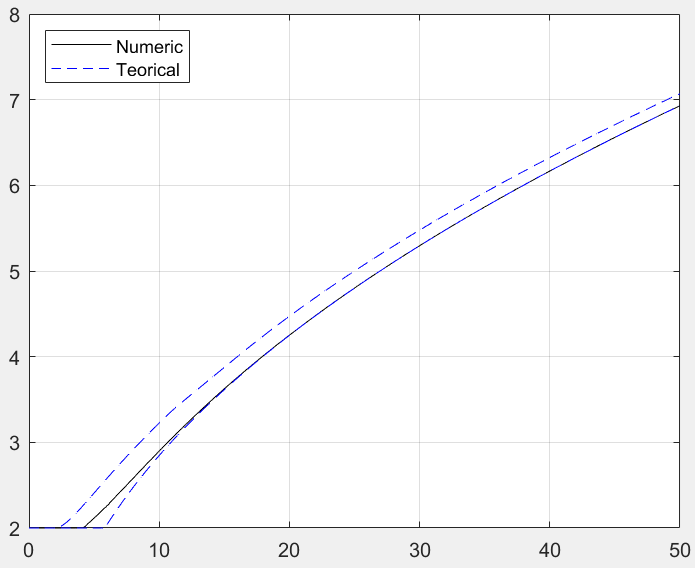}}
	{\includegraphics[width=1.85in]{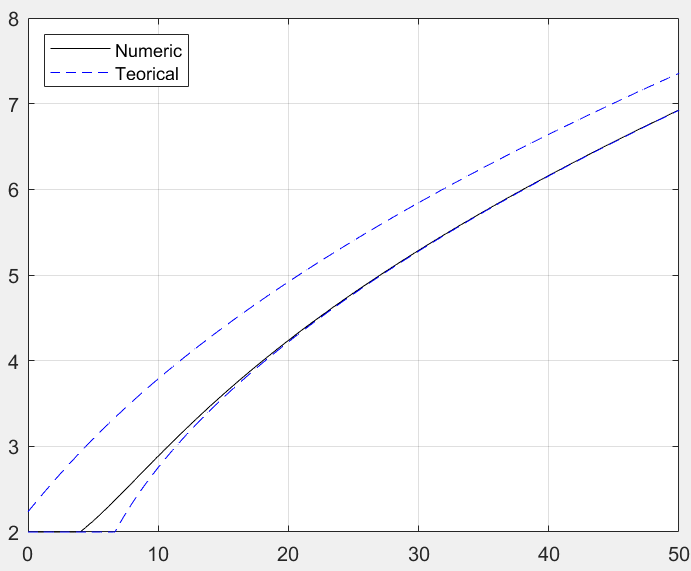}}
	{\includegraphics[width=1.85in]{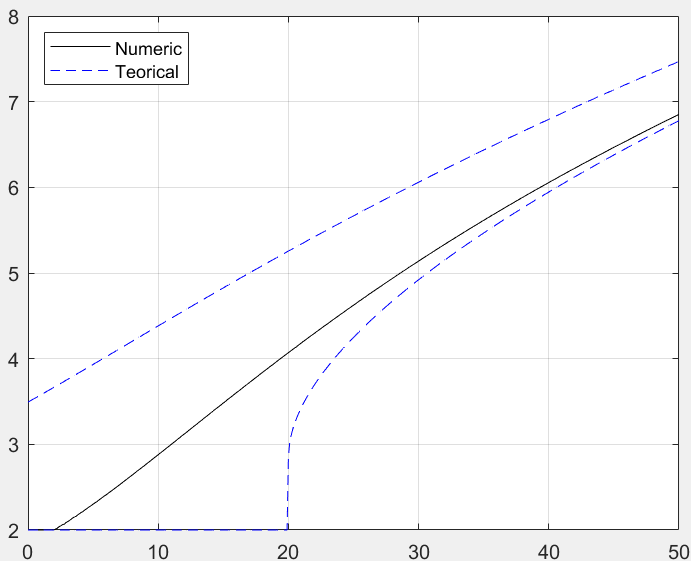}}
	\caption{This figure represents the theoretical values stated in Theorem \ref{teorema1} regarding the bounds of \(\sigma^*\) and the numerical approximation of the value of \(\sigma^*\) obtained in Octave. The figures illustrate different scenarios for various values of \(b\): Left-side figure: \(b = 0\);  Center figure: \(b = 5\); and Right-side figure: \(b = 40\).
		These figures provide a visual comparison between the theoretical bounds and numerical approximations of \(\sigma^*\) under varying conditions of \(b\).
	}
	\label{figF1}
\end{figure}

The previous theorem enables us to extend our analysis to the case where \(\Lambda > 0\), particularly for small values of \(\Lambda\). We employ perturbative methods to show that the values of \(\sigma\) identified in Theorem \ref{teorema1}, which guaranteed the existence of a solution in the unperturbed case, remain valid when \(\Lambda\) is small but nonzero. This ensures the existence of solutions in this modified setting as well.

\begin{theorem} \label{teorema2}
For each \(\sigma > \sigma^*(a,b)\) as defined in Theorem \ref{teorema1}, there exists a sufficiently small \(\Lambda_0 = \Lambda_0(\sigma, a, b) > 0\) such that, for any \(\Lambda \in (0, \Lambda_0)\), the problem \eqref{modeloriginal} admits a solution.	
\end{theorem}

Lastly, we will address the general case $\Lambda > 0$. To simplify the two differential equations, we will use the fundamental solution of the Helmholtz equation, allowing us to reduce them to a single non-local differential equation. By applying Leray-Schauder topological degree techniques, we will demonstrate the existence of traveling waves, as outlined in the following result.

\begin{theorem} 	\label{Teorema3} 
	If $\frac{b}{2\Lambda^2}<1$, there exists $\sigma\geq2$ such that the problem \eqref{modeloriginal} has solutions $\phi,V:\mathbb{R}\rightarrow\mathbb{R}$. Furthermore, we have the following properties:
	\begin{itemize}
		\item[$1.$] $\phi$ is strictly monotone.
		\item[$2.$] $\sigma\leq 2+\frac{a}{\Lambda}+\frac{b}{4\Lambda^2}\frac{2+\frac{a}{\Lambda}}{1-\frac{b}{2\Lambda^2}}$.
	\end{itemize}
\end{theorem}
\begin{remark}
	Additionally, it is shown that \(\phi^\prime(t) \to 0\) as \(|t| \to \infty\) in the three preceding theorems. Furthermore, both \(V\) and \(V^\prime\) exhibit this same behavior in Theorems \ref{teorema2} and \ref{Teorema3}.
\end{remark}

The paper is organized as follows. In Section 2, we investigate the existence of traveling waves in the special case where \(\Lambda = 0\). This simplifies the problem to a single second-order nonlinear differential equation. By transforming it into a first-order system and employing  upper and lower solution techniques, we prove Theorem \ref{teorema1} and the estimates of the threshold value $\sigma^*$. 
In Section 3, we extend the values of \(\sigma\) identified for the case \(\Lambda=0\) to small positive \(\Lambda\). This extension is accomplished through the use of geometric singular perturbation theory, which enables us to establish the existence of solutions for these small \(\Lambda\) values. 
In Section 4, we address the general case without setting any parameters to zero or consider small values of $\Lambda$. We apply a truncation argument, reducing the problem to a boundary value problem on a finite interval, and substitute the Fisher term with a combustion term that converges to the Fisher term in the limit. For this modified problem, we analyze the existence of a solution within a bounded domain using topological degree theory and fixed-point theory. Finally, we take the limit to prove the existence of traveling waves for the original problem.

\section{Local Advection term}
In this section, we will prove Theorem \ref{teorema1}. The case \(\Lambda=0\) simplifies the problem \eqref{modeloriginal}, reducing it to a second-order differential equation.
\begin{equation}
	\begin{array}{l}
		{-\sigma\phi^\prime-a(\phi^\prime)^2=(1+b(\phi^\prime)^2)\phi^{\prime\prime}+\phi(1-\phi)}, \\
		{\phi(-\infty)=1, \quad \phi(+\infty)=0.}
	\end{array}
	\label{probi2}
\end{equation}
We will reduce the study of the existence of solutions of \(\eqref{probi2}\) to the study of the first-order equation:
\begin{equation}
	\begin{aligned}
		&S^\prime=\frac{\sigma S-aS^2-\phi(1-\phi)}{S(1+bS^2)},\\
		&S(0)=S(1)=0,\; S(\phi)>0 \;,\phi\in(0,1),
	\end{aligned}
	\label{prob3}
\end{equation}
where $S:[0,1]\to [0,\infty)$ satisfies $S(\phi)=-\phi^\prime$. 
\begin{prop}
	There exists a monotone decreasing heterocline solution to the problem \eqref{probi2} if and only if there exists \( S \in C[0,1] \cap C^1(0,1) \) that satisfies \eqref{prob3}.
	\label{sisolosi}
\end{prop}

\begin{proof}
	If \(\phi\) is a solution of (\ref{probi2}), we can show that \(\phi'(s) < 0\). This allows us to define \(S \in C^1(0,1)\). Such reduction principles are frequently employed (see \cite{MalagutiMarcelli}) and their proof is standard.
\end{proof}

First, we will establish an a priori bound on the wave speed \(\sigma\), showing that if \(\sigma < 2\), no monotonically decreasing solutions exist. This is due to the fact that when \(\sigma < 2\), the solutions around \(\phi = 0\) exhibit oscillatory behavior. This is the central idea behind the following proposition.
\begin{prop}
If \(\sigma < 2\), then there are no solutions to \eqref{prob3}.
	\label{inferior}
\end{prop}
\begin{proof}		
	Let \( S(\phi) \) be a solution of (\ref{prob3}). According to Proposition \ref{sisolosi}, \( S(\phi) = -\phi^\prime \) for \(\phi \in (0,1)\), and \(\phi^\prime(t) \to 0\) as \(\phi(t) \to 0\), that is, as \( t \to +\infty \).
	
	Let us define the function $r(t)=-\frac{\phi^\prime(t)}{\phi(t)}=\frac{S(\phi(t))}{\phi(t)}$. This function $r(t)$ satisfies the differential equation:
	$$r^\prime=\frac{S^\prime\phi^\prime\phi-\phi^\prime S}{\phi^2}=r^2-S^\prime r.$$
	Since \( S \) is a solution of (\ref{prob3}), it follows that 
	$$r^\prime=r^2-\left( \frac{\sigma-aS}{1+bS^2}\right) r+\frac{1-\phi}{1+bS^2}.$$
	holds. We can observe that \( r' \) is expressed as a second-degree polynomial in \( r \). Given the convexity of this polynomial, we obtain 
	$$r^\prime\geq \frac{1-\phi}{1+bS^2}-\left( \frac{\sigma-aS}{2(1+bS^2)}\right)^2\rightarrow1-\frac{\sigma^2}{4},$$
	as $t\rightarrow+\infty$. Therefore, if $\sigma<2$, then $r$ is eventually decreasing and has limit. But this leads to a contradiction: If  $r(t)$ is bounded and  $r(t)\rightarrow\bar{r}\in\mathbb{R}$, then by considering a sequence $\left\lbrace t_n\right\rbrace_n$ tending to infinity where $r^\prime(t_n)\rightarrow0$, we derive
	$$\bar{r}^2-\sigma\bar{r}+1=0.$$
	This implies that \(\bar{r}\) would need to be a root of this second-degree polynomial. However, since \(\sigma < 2\), the polynomial has no real roots, leading to a contradiction.
		
On the other hand, if \( r(t) \rightarrow +\infty \), then \( \frac{r^\prime(t)}{r^2(t)} \rightarrow 1 \), and there exists a constant \( C > 0 \) such that \( \frac{r^\prime(t)}{r^2(t)} \geq C \) for large \( t \). This implies that \( r(t) \) must be larger than the solution to the differential equation \( y' = Cy^2 \). However, the solution to this equation blows up in finite time.

In summary, we have demonstrated that if \( \sigma < 2 \), then the problem \eqref{prob3} admits no solution.
\end{proof}

To establish the existence of a solution, we will use the following result, which relies on the concept of finding appropriate functions (subsolutions) that help control the evolution of solutions to \eqref{prob3} and, consequently, demonstrate the existence of a solution.
\begin{lemma}
	Let $\bar{S}\in C^1([0,1])\cap C((0,1))$  satisfying:
	\begin{equation}
		\begin{aligned}
			&\bar{S}^\prime(\phi)<\frac{\sigma \bar{S}(\phi)-a\bar{S}(\phi)^2-\phi(1-\phi)}{\bar{S}(\phi)(1+b\bar{S}^2(\phi))}, \;\phi\in(0,1)\\
			&\bar{S}(0)=\bar{S}(1)=0,\; \bar{S}(\phi)>0 \;\phi\in(0,1).
		\end{aligned}
		\label{subsolucion}
	\end{equation}
Then, there exists a solution \( S \) to the problem \eqref{prob3}.
	\label{teoremasub}
\end{lemma}

\begin{proof}
	
	Consider \( Z \) satisfying \eqref{subsolucion}. Let us construct the sequence of functions \(\{S_n\}_{n \geq 1}\), where \( S_n \) is the maximal solution of the initial value problem:
	\begin{equation}
		\begin{array}{l}
			S^\prime=\frac{\sigma S-aS^2-\phi(1-\phi)}{S(1+bS^2)},\;\phi\in(0,1-1/n),\\
			S(1-1/n)=Z(1-1/n).
		\end{array}
	\end{equation}	
The sequence constructed in this way satisfies \( S_{n+1}(\phi) \leq S_n(\phi) \) for all \(\phi \in (0,1-1/n)\) and \(n \in \mathbb{N}\). Furthermore, the  inequalities  \(0 < S_n(\phi) < Z(\phi)\) hold,  for all \(\phi \in (0,1-1/n)\) and \(n \in \mathbb{N}\).

By adapting the proof scheme outlined in \cite[Theorem 2.1]{MalagutiMarcelli}, we obtain the desired result. 
This version is already quite clear and concise, so only minor adjustments were made for readability.
\end{proof}

As a consequence of the monotone dependence of equation \eqref{probi2} on \(\sigma\), the solutions exhibit an ordered structure for different values of \(\sigma\). We can thus establish the following result:

\begin{prop} The set of admissible values
	\begin{equation}
		\left\lbrace \sigma>0: \text{such that \eqref{prob3} has solution}\right\rbrace
		\label{intervalo}
	\end{equation}
	 forms a closed, upper unbounded interval. The minimum of this interval is denoted by \(\sigma^* := \sigma^*(a, b)\).
	\label{monotonia}
\end{prop}
\begin{proof}
	First, we will prove that \eqref{intervalo} is an open, upper unbounded interval. Let us consider a solution \( S_1 \) of \eqref{prob3} with \(\sigma = \sigma_1\). If we take \(\sigma_2 \geq \sigma_1\), then \( S_1 \) serves as a lower solution to the problem \eqref{prob3} with \(\sigma = \sigma_2\). According to Lemma \ref{teoremasub}, this implies the existence of a solution for \(\sigma_2\). Hence, if there exists a solution for \(\sigma_1\), then there exists a solution for all \(\sigma \in (\sigma_1, +\infty)\).
	
	This allow us to define \(\sigma^*:=\sigma^*(a,b)\) as the infimum of this interval. Furthermore, by Proposition \ref{inferior}, we have \(\sigma^* \geq 2\). Now, we need to show that the interval \eqref{intervalo} is also closed.
	
	Assume \(\sigma > \sigma^*\). Then, we know that there exist \(\bar{\sigma}\) and \(\hat{\sigma}\) such that the problem \eqref{probi2} has a solution for \(\sigma^* < \hat{\sigma} < \sigma < \bar{\sigma}\). Using Proposition 1, let \(\bar{S}\) and \(\hat{S}\) be the solutions of problem \eqref{prob3} associated with \(\bar{\sigma}\) and \(\hat{\sigma}\), respectively. Due to the monotonicity of \eqref{prob3} with respect to \(\sigma\), \(\bar{S}\) and \(\hat{S}\) are subsolution and supersolution of \eqref{prob3}, respectively.
	
	Let \(m = \min\{\bar{S}, \hat{S}\}\) and \(M = \max\{\bar{S}, \hat{S}\}\). Consider the sequence of compact intervals \(\left\lbrace A_n \right\rbrace _n \subset (0,1)\) whose union covers the interval \((0,1)\). For each \(n\), we can find a solution \(S^1_n(\phi)\) to the problem \eqref{prob3} such that \(m(\phi) \leq S^1_n(\phi) \leq M(\phi)\), for \(\phi \in A_n\).
	
	Let \(S_n\) be the maximal solution of \eqref{prob3} in \(A_n\), extended to the entire interval \((0,1)\) by a constant. Then, \(S_n \geq S_{n+1}\) in \(A_n\), and the sequence \(\left\lbrace S_n \right\rbrace_n\) converges to a function \(S_0\) on \((0,1)\) as \(n \rightarrow \infty\). Moreover, this convergence is uniform on compact sets of \((0,1)\) because \(S^\prime_n\) is bounded on \((0,1)\). Let us analyze what happens at the endpoints of the interval \((0,1)\). We know that \(0 < m(\phi) \leq S^1_n(\phi) \leq M(\phi)\) for \(\phi \in (0,1)\), and that the functions \(m(\phi)\) and \(M(\phi)\) are continuous on \([0,1]\) and vanish at the boundary. Therefore, \(S_0(\phi)\) will also vanish at the boundary.
	
	We have thus proved that \(S_0\) is a solution of \eqref{prob3}, implying the existence of a solution for \eqref{probi2} when \(\sigma > \sigma^*\).
	
	Now, we only need to analyze the case \(\sigma = \sigma^*\). To handle this, define a decreasing sequence \(\left\lbrace \sigma_n \right\rbrace _n\) converging to \(\sigma^*\) with \(\sigma_n \leq \bar{\sigma}\) for all \(n \in \mathbb{N}\). Let \(\left\lbrace S_n \right\rbrace _n\) be the sequence of solutions to \eqref{prob3} for \(\sigma = \sigma_n\).
	
	Define \(S_0 = \inf_{n \in \mathbb{N}} S_n(\phi)\) for \(\phi \in [0,1]\), and verify that \(\inf_{\phi \in C} S_0(\phi) > 0\) for every compact subset \(C \subset (0,1)\). Assume, by contradiction, that there exists a compact set \(C \subset (0,1)\) such that \(\inf_{\phi \in C} S_0(\phi) = 0\). In other words, there exists a sequence \(\left\lbrace \phi_n \right\rbrace _n\) in \(C\) converging to \(\phi^* \in C\) such that \(S_0(\phi_n) < \frac{1}{n}\). Additionally, we can find a sequence \(\left\lbrace k_n \right\rbrace _n\) such that
	\[ S_{k_n}(\phi_n) < \frac{1}{n}, \quad n \in \mathbb{N}. \]
	
	Let \(I \subset (0,1)\) be a compact interval such that \(C \subset \text{int}(I)\). We know that if the solution to problem \eqref{prob3} approaches zero, it can only do so in the region where \(S^\prime < 0\). Therefore, searching for \(S_{k_n}(\phi_n) < \frac{1}{n}\), we have \(S^\prime_{k_n}(\phi) < 0\) for all \(\phi \in [\phi_n, \max I]\). This is a contradiction since \(S_{k_n}(\phi_n) \rightarrow 0\), and \(\left\lbrace \phi_n \right\rbrace _n\) converges to \(\phi^* < \max I\), but \(S_{k_n}\) is a solution of \eqref{prob3} and satisfies \(S_{k_n}(\phi) > 0\) for all \(\phi \in (0,1)\).
	
	Thanks to the continuity and differentiability of the solutions of \eqref{prob3}, the sequence \(\left\lbrace S_n \right\rbrace _n\) is bounded on any compact set \(C \subset (0,1)\), implying that it is Lipschitz on any compact set \(C\).
	
	Let \(\left\lbrace C_k \right\rbrace _k\) be a sequence of increasing compact subsets in \((0,1)\) whose union is a covering of \((0,1)\). For \(k = 1\), there exists a subsequence \(\left\lbrace S_{k_n^1} \right\rbrace _n\) that converges uniformly to a function \(S^1\) on \(C_1\), which is continuous and positive on \(C_1\). If we now consider \(k = 2\), the subsequence \(\left\lbrace S_{k_n^1} \right\rbrace _n\) admits a further subsequence \(\left\lbrace S_{k_n^2} \right\rbrace _n\) that converges uniformly to a function \(S^2\) on \(C_2\), which is continuous and positive on \(C_2\), coinciding with \(S^1\) on \(C_1\). By continuing this diagonal extraction procedure, we can define a function \(S^*\) that is continuous and positive on \((0,1)\), coinciding with \(S^n\) on \(C_n\) for every \(n \in \mathbb{N}\), and \(S^*\) is the uniform limit of the sequence of solutions of the equation \eqref{prob3} related to the problem \((P_{c_n})\).
	
	Therefore, taking the limit, \(S^*\) is a solution of \eqref{prob3} for \(\sigma = \sigma^*\) on \((0,1)\). We have \(S^*(\phi) > 0\) for \(\phi \in (0,1)\), and if we take \(\tilde{\sigma} > \sigma^*\), it holds that \(S_{\tilde{\sigma}}(\phi) \geq S_{\sigma^*}(\phi)\) for every \(\phi \in [0,1]\). Consequently, \(S^*(0) = S^*(1) = 0\).

	In particular, this demonstrates that \(\sigma^*\) is the infimum of the interval \eqref{intervalo}.
\end{proof}
 
 Let us now conclude the proof of statement (T1) in Theorem \ref{teorema1}. To achieve this, we will utilize Lemma \ref{teoremasub} for a function of the form $S(\phi)=\alpha\phi(1-\phi)$.
 
 \begin{prop}
 	If $a<2$, $b<2$, and $\sigma=2$, then there exists a solution to \eqref{probi2}.
 	\label{casob2}
 \end{prop}
 
 \begin{proof}
 	We want to verify that \( S(\phi) = \alpha \phi(1 - \phi) \) satisfies \eqref{subsolucion} for some \(\alpha \in \mathbb{R}\). To do this, we need to check that the following condition holds:
 	$$
	\frac{\sigma\alpha\phi(1-\phi)-a\alpha^2\phi^2(1-\phi)^2-\phi(1-\phi)}{\alpha\phi(1-\phi)(1+b\alpha^2\phi^2(1-\phi^2))}>\alpha-2\alpha\phi,
	$$ 
 	for $\phi\in(0,1)$. Rearranging, we arrive at the following expression 	
 	 $$2\phi^5-5\phi^4+4\phi^3+\frac{a-b\alpha^2}{b\alpha^2}\phi^2+\frac{(2-a)}{b\alpha^2}\phi+\frac{\sigma\alpha-1-\alpha^2}{b\alpha^4}>0, $$
 	for  $\phi\in(0,1)$. The polynomial $2\phi^5-5\phi^4+4\phi^3$ is greater than zero, for $\phi\in(0,1)$. It is sufficient to study when
 	$$\frac{a-b\alpha^2}{b\alpha^2}\phi^2+\frac{(2-a)}{b\alpha^2}\phi+\frac{\sigma\alpha-1-\alpha^2}{b\alpha^4}>0, \text{ for } \phi\in(0,1).$$ 
 	
 	Taking $\alpha=1$ and $\sigma=2$, we only need to verify the condition when
 	$$\phi((a-b)\phi+(2-a))$$
 	is positive for  $\phi\in(0,1)$, This requires  $a<2$ and $b<2$. Therefore, Lemma \ref{teoremasub} leads to the existence of a solution.
 \end{proof}

The next step is to establish the upper bounds for \(\sigma^*(a,b)\). We will start by examining the bounds given in statement (T2). Utilizing Lemma \ref{teoremasub} and the solutions from the FKPP model, we will derive the following result.
\begin{prop}
	Let $b=0$ and 
	\begin{equation}
		\sigma>\left\{\begin{array}{ll}
			2+\frac{a}{8} & \text{si  }\; a\leq16, \\
			\sqrt{a} & \text{si }\; a>16,
		\end{array}\right.
		\label{conds}
	\end{equation}
	Then, there exists a solution to \eqref{probi2}.
	\label{cotab01}
\end{prop}
\begin{proof}
	
	Let $S_F$ be a solution of
	\begin{equation}
		\begin{array}{l}
			S^{\prime}=\frac{cS-\phi(1-\phi)}{S}, \\
			S(0)=S(1)=0,\; S(\phi)>0,\; \phi \in(0,1).
			\label{fishe}
		\end{array}
	\end{equation}
	It is well known that this first-order equation corresponds to the problem
	\begin{equation}
		\begin{array}{l}
			{-c\phi^\prime=\phi^{\prime\prime}+\phi(1-\phi)}, \\
			{\phi(-\infty)=1, \quad \phi(+\infty)=0,}
		\end{array}
	\end{equation}
	 whose differential equation is the FKPP equation. It is well established that solutions exist for \(c \geq 2\).	Moreover, $S_F(\phi)<\frac{1}{4c}$ is satisfied, for $\phi\in(0,1)$. 
	
	It can be checked that \eqref{conds} is equivalent to 
	$$
	a\leq\left\{\begin{array}{ll}
		8\sigma-16 & \text{if  }\; \sigma<4, \\
		\sigma^2 & \text{if }\; \sigma\geq 4.
	\end{array}\right.
	$$
	Let us demonstrate that, under these conditions, there always exists \(c \geq 2\) that satisfies the inequality
	$$	S_F^\prime<\frac{\sigma S_F-aS_F^2-\phi(1-\phi)}{S_F},$$
	for $\phi\in(0,1)$. By using Lemma \ref{teoremasub}, the expression \eqref{fishe} and the fact that $S_F(\phi)<\frac{1}{4c}$, it is enough to prove that	
	$$cS<\sigma S-aS^2, \; \text{ for } S\in\left( 0,\frac{1}{4c}\right). $$
	Or equivalently, if there exists $c\geq2$ satisfying the above inequality for $S=\frac{1}{4c}$, i.e. , if it verifies
	\begin{equation}
		c^2-\sigma c+ \frac{a}{4}<0.
		\label{parabola1}
	\end{equation}
	If $\sigma<4$, then \eqref{parabola1} holds, for $c\geq 2$, if
	$$\frac{\sigma+\sqrt{\sigma^2-a}}{2}>2,$$
	which is true for all $a\leq 8\sigma-16$.
	
	If $\sigma\geq4$, then \eqref{parabola1} holds, for $c\geq 2$, if $a<\sigma^2$.
\end{proof}

The other bounds on $\sigma^*$, both for $b=0$ and $b>0$, rely on finding an $\alpha\in(0,\infty)$ such that $S(\phi)=\alpha\phi(1-\phi)$ satisfies \eqref{subsolucion}. We have the following result:

\begin{prop}
	Let us assume that we are in one of the following situations:
	\begin{enumerate}
		\item $b=0$, $a>2$ and $\sigma>\sqrt{\frac{4+a^2}{a}}$.
		\item $b>0$ and $\sigma>\sqrt{\frac{\sqrt{a^2+8a+4b+16}+a+4}{2}}$.
	\end{enumerate}
	Then, there exists a solution of \eqref{probi2}.
	\label{cotab03}
\end{prop}
\begin{proof}
	Let $S=\alpha \phi(1-\phi)$, for a certain $\alpha>0$. We want to determine the values of \(\alpha\) that satisfy the following inequality:
	\begin{equation}
		\frac{\sigma\alpha\phi(1-\phi)-a\alpha^2\phi(1-\phi)-\phi(1-\phi)}{\alpha\phi(1-\phi)(1+b\alpha^2\phi^2(1-\phi^2))}>\alpha-2\alpha\phi,
		\label{cotasub}
	\end{equation}
	for $\phi\in(0,1)$. 
	
	Consider $b=0$. Reordering \eqref{cotasub}, we obtain	$$F(\phi)=a\alpha^2\phi^2+(2\alpha^2-a\alpha^2)\phi+(\sigma\alpha-\alpha^2-1)>0, \text{ for }\in(0,1)$$
	We need to determine the conditions under which the quadratic function \( F(\phi) \) takes positive values for \(\phi \in [0,1]\).
		
	If $a>2$, wwe require that   $F(\phi)$ has a negative discriminant, obtaining
	$$\alpha^2(a^2+4)-4a\sigma\alpha+4a<0.$$	
	
	Again, we have a quadratic function in $\alpha$, then the existence of values of $\alpha$ verifying such equation reduces to the computation of the discriminant:
	$$a\sigma^2-(a^2+4)>0.$$
	This yields the condition on \(\sigma\), which ensures that \(F(\phi)\) is positive for \(\phi \in [0, 1]\).
	
	Let us now consider $b>0$. To prove \eqref{cotasub}, it is sufficient to impose
	$$\frac{\sigma\alpha-a\alpha^2\phi(1-\phi)-1}{\alpha(1+b\alpha^2\phi^2(1-\phi^2))}>\alpha,$$ 
	for $\phi\in]0,1[$.
	The expression above can be simplified by rewriting it as a function of \( z = \phi(1 - \phi) \). After substituting \( z \) and performing the necessary simplifications, we obtain:
	$$b\alpha^4z^2+a\alpha^2z+1-\sigma\alpha+\alpha^2<0, $$
	for $z\in[0,1/4]$.
	
	It can be observed that this expression is an increasing function of $z$. Thus, the only thing to check is the existence of a value $\alpha>0$ such that
	$$\frac{b\alpha^4}{16}+\frac{a\alpha^2}{4}+1-\sigma\alpha+\alpha^2<0.$$
	
	The expression provided in the Proposition is derived by setting \(\alpha = \frac{2}{\sigma}\). By applying Lemma \ref{teoremasub}, we establish the existence of a solution.
\end{proof}

To complete our analysis, we now focus on determining the lower bounds of \(\sigma^*\). For this, we use a concept analogous to that in Lemma \ref{teoremasub}, but our aim is to identify a function that facilitates the demonstration of non-existence of solutions. We present the following result:
\begin{prop}
	There is no solution to \eqref{probi2}  if any of the following conditions hold:
	\begin{itemize}
		\item $b=0$, $a\geq 3+2\sqrt{2}$ and $\sigma<\frac{(a-1)}{\sqrt{a}}$.
		\item $b>0$, $a>2$, $a\geq 2\sqrt{b}$ and $\sigma<\frac{2b-(a-\sqrt{a^2-4b})}{\sqrt{2b\left( a-\sqrt{a^2-4b}\right) }}$.
	\end{itemize}
\label{propcota}
\end{prop}

\begin{proof}
	Let us consider the problem (\ref{probi2}) and rewrite it in terms of $\psi=1-\phi$. 
	\begin{equation}
		\begin{array}{l}
			{\sigma\psi^\prime-a(\psi^\prime)^2=-(1+b(\psi^\prime)^2)\psi^{\prime\prime}+\psi(1-\psi)}, \\
			{\psi(-\infty)=0, \quad \psi(+\infty)=1.}
		\end{array}
	\end{equation}
		The corresponding first order equation is
	\begin{equation}
			J^\prime=\frac{aJ^2-\sigma J+\psi(1-\psi)}{J(1+bJ^2)}.
		\label{eqdemos}
	\end{equation}
	In the vicinity of 0, this equation exhibits two real eigenvalues with opposite signs. Consequently, the unstable manifold corresponds to the eigenvalue \(\lambda = \frac{-\sigma + \sqrt{\sigma^2 + 4}}{2}\).
	
	Let \( J(\psi) = \lambda\psi - \mu\psi^2 \), where \(\mu\) is chosen such that
\[
J^\prime < \frac{aJ^2 - \sigma J + \psi(1 - \psi)}{J(1 + bJ^2)},
\]
for \(\psi \in (0,1)\). We aim to demonstrate that any solution \( S \) of the problem (\ref{eqdemos}), initiated from the unstable manifold associated with the eigenvalue \(\lambda\), will satisfy \( S(\psi) \geq J(\psi) \) for \(\psi \in (0,1)\) and also \( S(1) > \frac{\lambda}{\mu} \geq 1 \). Consequently, solutions to \eqref{eqdemos} cannot satisfy \( S(0) = S(1) = 0 \), indicating that the problem \eqref{probi2} has no solution.
	
Consider \( S_n \) as the solution of \eqref{eqdemos} with initial condition \( S_n(0) = \frac{1}{n} \). It follows that \( S_n(\psi) \) is a decreasing sequence for all \( n \geq 1 \) due to the uniqueness of the solution.

Define \( R_n(\psi) = S_n(\psi) - J(\psi) \). We have \( R_n(0) > 0 \) and \( R_n^\prime(\psi) > 0 \), which implies \( R_n(\psi) > 0 \) for \( \psi \geq 0 \). Hence, \( S_n(\psi) > J(\psi) \) for \( \psi \geq 0 \).

Consequently, the sequence \( \{ S_n(\psi) \}_{n \geq 1} \) and its derivative are uniformly bounded because \( J(\psi) < S_n(\psi) < S_1(\psi) \) for all \( \psi \in (0,1) \) and \( n \geq 1 \). Therefore, we can extract a convergent subsequence \( S_{n_k} \rightarrow \bar{S}(\psi) \), which satisfies \( \bar{S}(0) = 0 \), \( \bar{S}(1) > 1 \), and \(\bar{S} \) is a solution of \eqref{eqdemos} for \( \psi \in (0,1) \).

Moreover, due to the uniqueness of the solution from the stable manifold, the only solution emerging from \( \psi = 0 \) is \( \bar{S}(\psi) \). This completes the proof.	
	
	We are seeking values of \(\mu\) that satisfy the following inequality:
	$$
	\frac{a(\lambda \psi-\mu\phi^2)^2-\sigma(\lambda\psi-\mu\psi^2)+\psi(1-\psi)}{(\lambda\psi-\mu\psi^2)(1+b(\lambda\psi-\mu\psi^2)^2)}>(\lambda-2\mu\psi),
	$$
	for all $\psi\in(0,1)$, or equivalently
	\begin{eqnarray}
			\Big(\hspace{-0,3cm}&-&\hspace{-0,2cm}a \mu^{2} \psi^{4}+\left(2 \mu^{2}+2 a \lambda \mu\right) \psi^3-\left(3 \mu \lambda+a \lambda^{2}+\sigma \mu-1\right)\psi^2 +(\lambda^2 +\sigma\lambda-1)\psi\Big) \nonumber\\ &
			+&\hspace{-0,2cm}b\left(2 \mu^{4} \psi^{7}-7 \mu^{3} \lambda \psi^{6}+9 \mu^{2} \lambda^{2} \psi^{5}-5 \lambda^{3} \mu \psi^{4}+\lambda^{4} \psi^3\right) < 0,
	\label{polinomio}
	\end{eqnarray}
	for all $\psi\in(0,1)$.
	
	First we analyze the case \(b = 0\). We start by noting that \(\lambda\) was defined such a way it satisfies the quadratic equation:
\[
\lambda^2 + \sigma \lambda - 1 = 0.
\]
This implies that the linear term in \(\psi\) is zero. Consequently, the inequality that must be satisfied is:
	$$\psi^2\left( -a\mu^2\psi^2+(2\mu^2+2a\lambda\mu)\psi-(3\mu\lambda+a\lambda^2+\sigma\mu-1)\right)< 0.$$
	
To determine whether the inequality is satisfied, we need to evaluate the quadratic function at specific values of \(\psi\). We want to ensure that this function takes negative values at \(\psi = 1\) and that the maximum of the function occurs at a value \(\psi^*\) where \(\psi^* > 1\).	
	Let us check that for $\psi=1$ this is satisfied. This leads to
	$$(2-a)\mu^2+(2a\lambda-3\lambda-\sigma)\mu+(1-a\lambda^2)<0.$$

If \(1 - a\lambda^2 < 0\), then the inequality is satisfied for values of \(\mu\) close to zero. This condition \(1 - a\lambda^2 < 0\) follows from the hypothesis of the Proposition, and it can be verified through a straightforward calculation using \(\lambda = \frac{-\sigma + \sqrt{\sigma^2 + 4}}{2}\).
	
	Let us finally see that $\psi^*$ is greater than 1:
	$$\psi^*=\frac{2\mu^2+2a\lambda\mu}{2a\mu^2}>1 \Leftrightarrow \mu<\frac{a\lambda}{a-1}, 
	$$
	but this is satisfied, since $\mu\leq\lambda$ and $a>2$ by hypothesis.

	In the case $b>0$, the inequality to examine is
	\begin{equation}
		\begin{aligned}
			{\left[-a \mu^{2} \psi^{2}+\left(2 \mu^{2}+2 a \lambda \mu\right) \psi-\left(3 \mu \lambda+a \lambda^{2}+\sigma \mu-1\right)\right]}& \\
			+b\left[2 \mu^{4} \psi^{5}-7 \mu^{3} \lambda \psi^{4}+9 \mu^{2} \lambda^{2} \psi^{3}-5 \lambda^{3} \mu \psi^{2}+\lambda^{4} \psi\right] &< 0.
		\end{aligned}
	\end{equation}
	
	Moreover, the following inequality
	$$b\left[2 \mu^{4} \psi^{5}-7 \mu^{3} \lambda \psi^{4}+9 \mu^{2} \lambda^{2} \psi^{3}-5 \lambda^{3} \mu \psi^{2}+\lambda^{4} \psi\right]< b\psi\lambda^4 $$
	holds, for $\psi\in (0,1)$. This can be proven by taking $z=\frac{\mu}{\alpha}\psi$, for $z\in(0,\frac{\mu}{\lambda})\subset(0,1)$. We can then rewrite the above expression as follows:
\[ b \left[2 z^5 - 7 z^4 + 9 z^3 - 5 z^2 + z\right]  \leq bz, \]
which can be easily verified to hold.
	
	Therefore, it is sufficient to determine for which values of \(\mu\) the following inequality is satisfied:
	$$\left[-a \mu^{2} \phi^{2}+\left(2 \mu^{2}+2 a \lambda \mu+b\lambda^4\right) \phi-\left(3 \mu \lambda+a \lambda^{2}+\sigma \mu-1\right)\right]<0.$$
	
	It can be checked that the vertex of this parabola is always greater than one, and there always exists a $\mu$ in a neighborhood of 0 such that:
	$$2(1-a)\mu^2+2a\lambda\mu + b\lambda^4>0.$$
	
	Also, by evaluating the polynomial at $\psi=1$, the inequality
	$$(2-a)\mu^2+(2a\lambda-3\lambda-\sigma)\mu+(b\lambda^4-a\lambda^2+1)<0$$
	is satisfied in a neighborhood of 0, as 
	$$(b\lambda^4-a\lambda^2+1)<0$$
	holds for $\sigma$ given by the expression in the statement of the Proposition. 
\end{proof}

Once these intermediated results have been presented, we can complete the proof of Theorem \ref{teorema1}.
\begin{proof}[Proof of Theorem \ref{teorema1}]

Proposition \ref{monotonia} establishes the existence of \(\sigma^* := \sigma^*(a,b)\) such that, for all \(\sigma \geq \sigma^*\), there exists a traveling wave. From Proposition \ref{inferior}, we know \(\sigma^* \geq 2\). Furthermore, Proposition \ref{casob2} shows that \(\sigma^*(a,b) = 2\) if \(\max \{a, b\} \leq 2\). This confirms $(T1)$.

For the case \( b = 0 \) (i.e., $(T2)$), Proposition \ref{propcota} provides that \(\sigma^* \geq \frac{a - 1}{\sqrt{a}}\), given that \( a \geq 3 + 2\sqrt{a} \).

On the other hand, for the study of upper bounds, we have Proposition \ref{cotab01} and Proposition \ref{cotab03}. Let us determine the optimal upper bound estimates for $\sigma^*(a,0)$. For $a > 16$, the best estimate is clearly $\sigma(a,b) > \sqrt{a}$.

Let us examine the case $a \in (2, 16)$. By equating the estimates in this interval, $\sqrt{\frac{a^2 + 4}{a}} = 2 + \frac{a}{8}$, we obtain the following equation:
$$
p(a) = a^3 - 32a^2 + 256a - 256 = 0.
$$
From this, we have $p(2) > 0$ and $p(16) < 0$. Therefore, there is a unique root $a^*$ in the interval $(2, 16)$, which implies a unique intersection point of the two graphs. Note that for $a \geq 16$, the function $\sqrt{\frac{a^2 + 4}{a}}$ exceeds $\sqrt{a}$.

We find the following estimates for $\sigma^*(a,0)$:
$$
\sigma^*(a,0) \leq \left\{
\begin{array}{ll}
	\sqrt{\frac{a^2 + 4}{a}} & \text{if } 2 \leq a \leq a^*, \\
	2 + \frac{a}{8} & \text{if } a^* \leq a \leq 16, \\
	\sqrt{a} & \text{if } a > 16.
\end{array}
\right.
$$

Finally, let us prove $(T3)$, which studies the case $b > 0$. Proposition \ref{cotab03} provides the upper bound for $\sigma^*$, and Proposition \ref{propcota} establishes that if $a^2 \geq 4b$, the lower bound for $\sigma^*$ is no longer 2, but $\sigma^* \geq \sqrt{\frac{\sqrt{a^2 + 8a + 4b + 16} + a + 4}{2}}$.
\end{proof}

\section{Singular Perturbation Theory (Small $\Lambda$)}

In this section, we will address the case where \(\Lambda > 0\) is small. Our approach involves leveraging the results obtained for \(\Lambda = 0\) and applying geometric singular perturbation theory, as developed by Fenichel \cite{Feniche}, to extend the existence of solutions to small values of \(\Lambda\). Specifically, if a parameter set \((\sigma, a, b)\) permits the existence of a subsolution as described in \eqref{subsolucion}, then traveling waves will also exist for sufficiently small values of \(\Lambda\).

Let us rewrite  \eqref{modeloriginal} as a first-order system in the following form:
\begin{equation}
	\begin{split}
		\phi' &= \psi, \\
		\psi' &= -\sigma\psi + V\psi - \phi(1-\phi), \\
		\Lambda V' &= W, \\
		\Lambda W' &= (1 + b\psi^2)V + a\psi - b\sigma\psi^2 - b\psi\phi(1-\phi).
	\end{split}
	\label{slow}
\end{equation}
By making the change of variable $\xi = \Lambda\eta$, we obtain:
\begin{equation}
	\begin{split}
		\dot{\phi} &= \Lambda\psi, \\
		\dot{\psi} &= \Lambda(-\sigma\psi + V\psi - \phi(1-\phi)), \\
		\dot{V} &= W, \\
		\dot{W} &= (1 + b\psi^2)V + a\psi - b\sigma\psi^2 - b\psi\phi(1-\phi),
	\end{split}
	\label{fast}
\end{equation} 
where we denote \(\frac{d}{d\xi}={\;}^\prime\) and \(\frac{d}{d\eta}=\dot{\;}\). The set of critical points of \eqref{fast} for \(\Lambda=0\) is defined by:
\begin{equation}
	M_0 := \left\lbrace (\phi,\psi,V,W)\in\mathbb{R}^4\;|\; W=0, V=\frac{\psi(-a+b\sigma\psi+b\phi(1-\phi))}{(1+b\psi^2)} \right\rbrace
	\label{M0}
\end{equation}
Note also that the flow of \eqref{slow} is confined to \(M_0\) when \(\Lambda = 0\).

The perturbation theory proposed in \cite{Feniche} provides a manifold \(M_\Lambda\), which is close to \(M_0\) in a sense that will be specified, and is invariant under the flow associated to \eqref{fast}. This manifold lies within a neighborhood of \(\Lambda = 0\), specifically \(\mathcal{O}(\Lambda)\). This framework allows us to study the problem \eqref{slow} restricted to the manifold \(M_\Lambda\). To apply this theory, we will use the version established in \cite{Jones}, which has been successfully employed by various authors, see for instance \cite{Akveld, gourley} and the references therein.

To apply the theorem, it is essential to verify that \( M_0 \) is normally hyperbolic. This requires demonstrating that the Jacobian matrix of the system described by \eqref{fast} at points on \( M_0 \) has as many eigenvalues with zero real part as the complementary dimension of the manifold $M_0$. Specifically, we need to show that exactly two eigenvalues have zero real part. The Jacobian of \eqref{fast} is given by:
\begin{equation*}
	\left[\begin{array}{c}
		\dot{\phi} \\
		\dot{\psi} \\
		\dot{V} \\
		\dot{W}
	\end{array}\right]=\left[\begin{array}{cccc}
		0 & \Lambda & 0 & 0 \\
		\Lambda(1-2\phi) & \Lambda(-\sigma+V) & \Lambda\psi & 0 \\
		0 & 0 & 0 & 1 \\
		-b\psi(1-2\phi) & 2b\psi V+a-2b\sigma\psi-b\phi(1-\phi) & (1+b\psi^2) & 0
	\end{array}\right]
	\left[\begin{array}{c}
		{\phi} \\
		{\psi} \\
		{V} \\
		{W}
	\end{array}\right].
\end{equation*}
At \(\Lambda = 0\), the eigenvalues of the Jacobian are \(\lambda = 0\) with multiplicity two, and \(\lambda = \pm \sqrt{1 + b\psi^2}\). Consequently, \(M_0\) is normally hyperbolic.

Let \( B_R \) represent the ball of radius \( R \) centered at the origin in \(\mathbb{R}^2\).

\begin{prop} {\rm \cite[Theorem 1]{Jones}}.
	If \(M_0\) is a normally hyperbolic manifold, then for any \(R > 0\), for any open interval \(I\) containing \(\sigma\), and for any \(k \in \mathbb{N}\), there exists a \(\Lambda_0 > 0\), depending on \(R\), \(I\), and \(k\), such that for all \(\Lambda \in \left(0, \Lambda_0\right)\), there exists a manifold \(M_{\Lambda}\), given by
\[
\begin{split}
M_{\Lambda} = \Big\{(\phi, \psi, V, W) \in \mathbb{R}^4 \mid & W = g(\phi, \psi, \Lambda, \sigma), \\& V = f(\phi, \psi, \Lambda, \sigma), \ (\phi, \psi) \in B_R, \sigma \in I\Big\},
\end{split}
\]
where \(f\) and \(g\) are functions in \(C^k\left(\overline{ B_R } \times \bar{I} \times \left[0, \Lambda_0\right]\right)\). This manifold \(M_{\Lambda}\) is locally invariant under the flow of the system \eqref{fast}.
	\label{teoremaFeniche}
\end{prop}

This result implies that for any compact subset of \(M_0\) that includes the critical points of interest, we can identify an invariant manifold for the system \eqref{fast} that lies within an \(\mathcal{O}(\Lambda)\) neighborhood of the manifold \(M_0\). Moreover, there exist functions \(f(\phi, \psi, \Lambda, \sigma)\) and \(g(\phi, \psi, \Lambda, \sigma)\) with any desired level of regularity. This allows us to reduce the order of the equations in our system and focus on studying the system in the form:
\begin{equation}
	\begin{split}
		\phi^\prime &= \psi, \\
		\psi^\prime &= -\sigma\psi + f(\phi,\psi,\Lambda,\sigma)\psi - \phi(1-\phi),
	\end{split}
	\label{perturbation}
\end{equation}

Since the functions \( f \) and \( g \) are in \( C^k\left(\overline{ B_R } \times \bar{I} \times \left[0, \Lambda_0\right]\right) \), they can be expanded in a Taylor series in \(\Lambda\) as follows:
\[ f(\phi, \psi, \Lambda, \sigma) = \sum_{n=0}^{k} f_n(\phi, \psi, \sigma) \Lambda^n + F(\phi, \psi, \Lambda, \sigma) \Lambda^k, \]
where \( F(\phi, \psi, \Lambda, \sigma) \) is a continuous function in \(\Lambda\) with \( F(\phi, \psi, 0, \sigma) = 0 \). A similar expansion applies to the function \( g \).

Observe that, by construction of $M_0$ \eqref{M0}, \( f_0 = \frac{\psi(-a + b\sigma \psi + b\phi(1 - \phi))}{1 + b\psi^2} \) and \( g_0 = 0 \). The remaining terms can be determined using the fact that the vector field of \eqref{fast} is perpendicular to the normals of \( M_\Lambda \). Substituting \( f = f_0 + \Lambda \bar{f} \), where \(\bar{f}\) is the remainder in the Taylor expansion of \(f\),  into \eqref{perturbation}, we obtain:
\begin{equation}
	\begin{split}
		&\phi^\prime =\psi,\\
		&\psi^\prime=\frac{-\sigma\psi-a\psi^2+\Lambda\bar{f}(\phi,\psi,\Lambda,\sigma)(1+b\psi^2)\psi-\phi(1-\phi)}{1+b\psi^2},\\
		&\phi(-\infty)=1,\; \phi(+\infty)=0.
	\end{split}
	\label{perturbation2}
\end{equation}

Let us now prove the following result, which encapsulates the essence of Theorem \ref{teorema2}.

\begin{prop}
	Let \((\sigma, a, b)\) be such that there exists a function \(H(\phi)\) satisfying the subsolution condition \eqref{subsolucion}. Then, there exists a \(\bar{\Lambda}\) such that the problem \eqref{perturbation2} has a solution for all \(\Lambda \in (0, \bar{\Lambda})\).
\end{prop}
\begin{proof}
	First, choose \( R > 0 \) such that the ball \( B_R \) encompasses the critical points \((0,0)\), \((1,0)\), and the subsolution \( S(\phi) \). We then apply Theorem \ref{teoremaFeniche} for this choice of \( R \).
Similarly to Proposition \ref{sisolosi} and Lemma \ref{teoremasub}, we need to demonstrate that equation \eqref{perturbation2} has a solution if there exists a function \(\bar{S} \in C^1([0,1]) \cap C((0,1))\) that satisfies the following conditions:
	\begin{equation}
		\begin{split}
			\bar{S}^\prime(\phi) &< \frac{\sigma \bar{S}(\phi) - a\bar{S}(\phi)^2 - \Lambda \bar{f}(\phi, \bar{S}(\phi), \Lambda, \sigma)(1 + b \bar{S}^2(\phi)) \bar{S}(\phi) - \phi(1 - \phi)}{\bar{S}(\phi)(1 + b \bar{S}^2(\phi))}, \\
			\bar{S}(0) &= \bar{S}(1) = 0, \quad \bar{S}(\phi) > 0.
		\end{split}
		\label{subsolucionpertur}
	\end{equation}

Now, consider \( H(\phi) \) as defined in the statement of the Proposition, whose properties we need to establish to complete our argument.  By the continuity of the function \(\bar{f}\) and applying Theorem \ref{teoremaFeniche}, we have \( |\bar{f}(\phi, H(\phi), \sigma, \Lambda)| < K \) for \(\phi \in (0,1)\) and \(\Lambda \in (0, \Lambda_0)\). Substituting \( H \) into \eqref{subsolucionpertur}, we obtain:
$$
	H^\prime(\phi) < \frac{\sigma H(\phi) - a H(\phi)^2 - \Lambda \bar{f}(\phi, H(\phi), \Lambda, \sigma) (1 + b H^2(\phi)) H(\phi) - \phi (1 - \phi)}{H(\phi) (1 + b H^2(\phi))}.
	$$
To demonstrate that \( H \) is a subsolution, it is sufficient to verify that:
$$
	H^\prime(\phi) < \frac{\sigma H(\phi) - a H(\phi)^2 - \Lambda K (1 + b H^2(\phi)) H(\phi) - \phi (1 - \phi)}{H(\phi) (1 + b H^2(\phi))},
	$$
	for $\phi \in (0,1)$. Equivalently,
	$$
	H^\prime(\phi) < \frac{\sigma H(\phi) - a H(\phi)^2 - \phi (1 - \phi)}{H(\phi) (1 + b H^2(\phi))} - \Lambda K,
	$$
	for $\phi \in (0,1)$. To prove this, it suffices to show that the following difference is uniformly bounded:
	$$ \frac{\sigma H(\phi) - a H(\phi)^2 - \phi (1 - \phi)}{H(\phi) (1 + b H^2(\phi))}-H^\prime(\phi)=\frac{\sigma-\sigma^*}{(1 + b H^2(\phi))}.$$
	Since \(H(\phi)\) is a solution to \eqref{prob3} for \(\sigma^*\), it follows that \(H(\phi)\) is uniformly bounded from below for all \(\phi \in (0,1)\). Consequently, there exists a \(\Lambda_0\) such that the inequality holds for all \(\Lambda \in (0, \Lambda_0)\). This guarantees that \(H(\phi)\) satisfies the required condition for a subsolution, thus establishing the existence of a solution and concluding the proof of the Proposition.
\end{proof}

A direct consequence of this result, assuming $H(\phi)$ as the solution to \eqref{prob3} for $\sigma=\sigma^*(a,b)$, is that it can be shown that for every $\sigma>\sigma^*(a,b)$, there exists a solution to \eqref{modeloriginal} for small values of $\Lambda$. Proving the Theorem \ref{teorema2} through this method.

\section{Non Local Advection Term ($\Lambda>0$)}

In this section, we will investigate the problem \eqref{modeloriginal} for general values of \(\Lambda > 0\). Unlike the small perturbation of a second-order system considered previously, this scenario involves a fourth-order system. As such, we must employ different techniques to establish the existence of traveling wave solutions. The methods used here are inspired by the work of various authors, including Berestycki, Henderson, Lions, Nadin, Perthame, and Ryzhik, as detailed in \cite{BerestyckiLions, BerestyckiPerthame, Henderson, PerthameNadin} and the references therein.

To prove Theorem \ref{Teorema3}, we will consider a modified problem that introduces new parameters. Several adjustments to the original problem \eqref{modeloriginal} are required for this purpose.

Firstly, we define a truncature function \( g(\phi) \) as follows: \( g(\phi) = 1 \) for \( \phi \geq 1 \), \( g(\phi) = 0 \) for \( \phi \leq \theta \), \( g'(\phi) \geq 0 \) for \( \phi \in (\theta, 1) \), and \( g(\phi) \to 1 \) as \( \theta \to 0 \). Note that \( g(\phi) \) is bounded by 1.

Additionally, it is necessary to work within a compact interval \([- \alpha, \alpha]\) and subsequently take the limit as \( \alpha \) approaches \( +\infty \).

In this way, we arrive at the following problem:
\begin{equation}
	\begin{aligned}
		& \phi'' + \sigma\phi' + g(\phi) u'\phi' + g(\phi) \phi(1-\phi) = 0,\quad \xi\in[-\alpha,\alpha], \\
		& \phi(-\alpha) = 1,\quad \phi(+\alpha) = 0,\quad \phi(0) = \theta,
	\end{aligned}
	\label{probih}
\end{equation} 
where
\begin{equation}
	u(\xi)=a(\Gamma * \bar{\phi})+\frac{b}{2} \left(\Gamma *(\bar{\phi}^\prime)^2\right),
	\label{ecuacionU}
\end{equation}
being	$\Gamma(\xi) = \frac{1}{2\Lambda}e^{-\frac{|\xi|}{\Lambda}}$, and $\bar{\phi}$  the extension by zero, for $\xi\in(\alpha,\infty)$ and by one, for $\xi\in(-\infty,\alpha)$. Later, \(V\) will be recovered as \(-u'\) after taking the limit of the parameters.

Let us begin by proving the positivity and monotonicity of the solution $\phi$.
\begin{lemma}
	Let $\phi:[-\alpha,\alpha]\rightarrow\mathbb{R}$ be a solution of \eqref{probih}, then $0\leq \phi\leq 1$ and $\phi^\prime<0$, for $\xi\in(-\alpha,\alpha)$.
	\label{propiedadessol}
\end{lemma}

\begin{proof}
	Let us prove that \(\phi(\xi) \leq 1\) for \(\xi \in (-\alpha, \alpha)\). We proceed by reductio ad absurdum.
Assume that \(\phi(\xi) > 1\) for some \(\xi\). Then, there exists a point \(\bar{\xi}\) such that \(\phi(\bar{\xi}) > 1\), \(\phi'(\bar{\xi}) = 0\), and \(\phi''(\bar{\xi}) \leq 0\).
However, since \(\phi\) is a solution of the differential equation, we have:
\[ \phi''(\bar{\xi}) = -g(\phi)\phi(\bar{\xi})(1 - \phi(\bar{\xi})). \]
Given that \(\phi(\bar{\xi}) > 1\) and \(g(\phi)\) is positive (as defined previously), the term \(\phi(\bar{\xi})(1 - \phi(\bar{\xi}))\) is negative. Consequently, 
\[ -g(\phi)\phi(\bar{\xi})(1 - \phi(\bar{\xi})) > 0, \]
implying that \(\phi''(\bar{\xi}) > 0\).
This is a contradiction because we assumed \(\phi''(\bar{\xi}) \leq 0\). Therefore, our assumption that \(\phi(\xi) > 1\) for some \(\xi\) must be false. Hence, \(\phi(\xi) \leq 1\) for all \(\xi \in (-\alpha, \alpha)\).

Let us analyze the monotonicity of the solution. Let \(\xi_0\) be the first value after \(-\alpha\) such that \(\phi'(\xi_0) = 0\). The existence of this point is possible by the uniqueness of the initial value problem (I.V.P.) for \(\phi'(-\alpha) \neq 0\). 

From the differential equation, we have:
\[ \phi''(\xi_0) = -g(\phi) \phi(\xi_0)(1 - \phi(\xi_0)) \leq 0. \]

If \(\phi(\xi_0) > \theta\), then \(g(\phi) > 0\) and thus \(\phi''(\xi_0) < 0\), which means \(\phi'(\xi)\) was increasing before reaching zero. This would imply that \(\phi'(\xi) > 0\) in a neighborhood to the left of \(\xi_0\), which contradicts \(\xi_0\) being the first point where \(\phi'(\xi) = 0\).

If \(\phi(\xi_0) \leq \theta\), then \(g(\phi) = 0\) because of the truncation function, making \(\phi''(\xi_0) = 0\). This would imply that \(\phi\) is a constant solution near \(\xi_0\). However, by the uniqueness of the I.V.P., if \(\phi\) were constant, it would contradict the existence of \(\xi_0\) as the first point where \(\phi'(\xi) = 0\) since \(\phi'(\xi)\) would not change sign.

Therefore, \(\phi\) must be strictly decreasing as it is not possible that \(\phi'(\xi) > 0\) for all \(\xi \in (-\alpha, \alpha)\). 

From this, we deduce that \(\phi(\xi) > 0\) for \(\xi \in (-\alpha, \alpha)\).
	\end{proof}

\begin{remark}
In the following results, \(\alpha\) will be considered fixed and will not be explicitly referenced unless it is necessary to indicate dependence on \(\alpha\). We will use the notation \(\|\cdot\|_{L^p(-\alpha, \alpha)} := \|\cdot\|_p\) for \(1 \leq p \leq \infty\) and \(\alpha \in (0, \infty)\).
\end{remark}

Let us start by analyzing the bounds of $u$ as a function of the norm of $\phi$ and $\phi^\prime$.
\begin{lemma}
	Let $\phi$ be a solution of \eqref{probih}, then we have
	$$\|u\|_{\infty}\leq a+\frac{b}{4\Lambda}\|\phi^\prime\|_{2}^2,$$
	$$\|u^\prime\|_{\infty}\leq\frac{a}{\Lambda}+\frac{b}{4\Lambda^2}\|\phi^\prime\|_{2}^2 .$$
	\label{cotauprima}
\end{lemma}
\begin{proof}
	As a result of the construction of the solution \( u \), and from the expression \eqref{ecuacionU}, we deduce that:
	\begin{align*}
		u(\xi) &= \int_{-\infty}^{+\infty} a\Gamma(x-y) \bar{\phi}(y) + \frac{b}{2} \Gamma(x-y) \bar{\phi}^{\prime^2}(y) \,dy \\
		&= a\int_{-\infty}^{+\infty} \Gamma(x-y) \bar{\phi}(y) + \frac{b}{2}\int_{-\alpha}^{\alpha} \Gamma(x-y) \phi^{\prime^2}(y) \,dy \\
		&\leq a\max_{y\in\mathbb{R}}\left\lbrace \bar{\phi}(y)\right\rbrace \int_{-\infty}^{+\infty} \Gamma(x-y) \,dy + \frac{b}{2}\max_{y\in\mathbb{R}}\left\lbrace \Gamma(x-y)\right\rbrace \int_{-\alpha}^{\alpha} \phi^{\prime^2}(y) \,dy.
	\end{align*}
	
	Lemma \ref{propiedadessol} states that \(\|\bar{\phi}\|_\infty \leq 1\). Combining this with \(\|\Gamma\|_\infty \leq \frac{1}{2\Lambda}\), we obtain:
	$$u(\xi)\leq a+\frac{b}{4\Lambda}\|\phi^\prime\|_{2}^2,$$
	for all $\xi \in [-\alpha,\alpha]$.
	
	Given that \(|\Gamma'(x)| = \frac{1}{\Lambda} \Gamma(x)\) for almost every \(x \in \mathbb{R}\), the bound for \(u'\) turns out to be
$$
\|u'\|_\infty \leq \frac{\|u\|_\infty}{\Lambda}.
$$
This establishes the second inequality of the Lemma and completes the proof.
\end{proof}

\begin{prop}
	Let $\phi$ be a solution to \eqref{probih}. Then, the following bound  
	$$
	\sigma< 2+\|u^\prime\|_\infty
	$$ 
	holds, for $\alpha>\alpha_0:=-\log(\theta)$.
	\label{cotainferior}
\end{prop}

\begin{proof}
	Let us demonstrate that 
$$
\sigma + u'(\xi) \leq 2 ,
$$ 
for some \(\xi \in [-\alpha, \alpha]\). By applying this result along with Lemma \ref{cotauprima}, we can establish the statement of the Proposition.
	
	The argument will proceed by reductio ad absurdum, assuming that \( u'(\xi) > 2 - \sigma \) for all \(\xi \in [-\alpha, \alpha]\).	
	Let $M$ be such that
	\begin{equation}
		\phi(\xi)<Me^{-\xi}, \;\xi\in[-\alpha,\alpha].
		\label{Psidif}
	\end{equation}
	Consider the function
	$$
	\Psi(\xi)=Me^{-\xi}-\phi(\xi),
	$$
	which satisfies
	\begin{equation}
		-\sigma\Psi^\prime-\Psi^{\prime\prime}- u^\prime\Psi^\prime= (\sigma-1+\tau u^\prime)\Psi- g(\phi)  \phi+g(\phi)\phi^2,
		\label{desigualdadpsi}
	\end{equation}
	where $\psi(\xi)=Me^{-\xi}$. From the reductio ad absurdum hypothesis, it follows that
\[ \sigma - 1 + g(\phi) u' > g(\phi). \]
Thus, the expression \eqref{desigualdadpsi} is strictly positive, indicating that no local minimum exists, as the second derivative at a critical point is negative.
	
A straightforward approximation argument reveals that if we choose \( M_0 \) as the minimum value for which the condition holds, then the function \( M_0 e^{-\xi} - \phi(\xi) \) achieves its global minimum at either \(\xi = -\alpha\) or \(\xi = \alpha\). By optimality, this minimum value must be zero, thereby excluding \(\xi = \alpha\).
 Taking $\xi=-\alpha$, we have 
	$$M_0e^{\alpha}-1=0.$$
	From this, we find that \( M_0 = e^{-\alpha} \). Additionally, we have
\[ 
M_0 e^{-\xi} - \phi(\xi) \geq 0,
 \]
which leads to a contradiction for \(\xi = 0\) if we assume \(e^{-\alpha} - \theta < 0\).
\end{proof}

Once we have determined the upper bound for \(\sigma\), we now proceed to find the lower bound.

\begin{prop}
	Under the hypotheses of Proposition \ref{cotainferior}, along with  \(\frac{b}{2\Lambda^2} < 1\) and \(\theta \in (0, 1/3)\), it follows that
	
	\begin{equation}
	(1-\frac{b}{2\Lambda^2})\int_{-\alpha}^{\alpha}|\phi^\prime|^2+\int_{-\alpha}^{\alpha}g(\phi)\phi(1-\phi)^2\leq 2+\frac{a}{\Lambda}+\frac{\theta}{\alpha},
	\label{desigualdadProp}
	\end{equation}
	and
	\begin{equation}
		\sigma\geq-\frac{6}{5}\frac{\theta}{\alpha},
		\label{cotasigmainferior} 
	\end{equation}
	for $\alpha>\alpha_0$.
	\label{propinf}
\end{prop}
\begin{remark}
Note that in the previous bounds, there is a term of the form \(\frac{\theta}{\alpha}\). This term can be bounded independently of \(\alpha\), specifically:
$$
\frac{\theta}{\alpha} \leq \frac{1}{3\alpha_0},
$$
where $\alpha_0$ was defined in Proposition \ref{cotainferior}.
\end{remark}

\begin{proof}
	Consider the differential equation associated with \(\phi\) given by \eqref{probih}. Multiplying both sides by \((1 - \phi)\), we obtain:
	$$
	\phi^{\prime\prime}(1-\phi)+\sigma\phi^\prime(1-\phi)+ g(\phi) u^\prime\phi^\prime(1-\phi)+g(\phi) \phi(1-\phi)^2=0.
	$$
	Rewriting the terms involving \(\phi''\) and \(\sigma\), we obtain the expression
	$$\left[ \phi^\prime(1-\phi)\right] ^\prime+\phi^{\prime2}-\frac{\sigma}{2}[(1-\phi)^2]^\prime+ g(\phi) u^\prime\phi^\prime(1-\phi)+ g(\phi) \phi(1-\phi)^2=0.$$
	Integrating this over the interval \((- \alpha, \alpha)\)
	$$
	\phi^\prime(\alpha)+\int_{-\alpha}^{\alpha}\phi^{\prime2}-\frac{\sigma}{2}+ \int_{-\alpha}^{\alpha} g(\phi) (1-\phi)\phi^\prime u^\prime+\int_{-\alpha}^{\alpha} g(\phi)\phi(1-\phi)^2=0.
	$$
	Rearranging the above expression, we find
	\begin{equation}
		\int_{-\alpha}^{\alpha}\phi^{\prime2} +\int_{-\alpha}^{\alpha} g(\phi)\phi(1-\phi)^2=\frac{\sigma}{2}-\phi^\prime(\alpha)-\int_{-\alpha}^{\alpha} g(\phi) (1-\phi)\phi^\prime u^\prime.
		\label{igualdadexp}
	\end{equation}
	Note that the terms on the left-hand side are strictly positive. We now proceed to bound the terms on the right-hand side of the equation.

Using the expression for \(u\) in \eqref{ecuacionU}, we can write \(u^\prime\) in \eqref{igualdadexp} as the sum of two terms, as follows:
	$$-\int_{-\alpha}^{\alpha} g(\phi) (1-\phi)\phi^\prime u^\prime=-\int_{-\alpha}^{\alpha} g(\phi) (1-\phi)\phi^\prime u^\prime_1-\int_{-\alpha}^{\alpha} g(\phi) (1-\phi)\phi^\prime u^\prime_2,$$
	where $u^\prime_1=a\Gamma*\phi^\prime$ y $u^\prime_2=\frac{b}{2}\Gamma^\prime*(\phi^{\prime2})$.
	Since \(\phi^\prime \leq 0\), it follows that \(u^\prime_1 \leq 0\), making the term involving \(u^\prime_1\) positive. Consequently, we obtain:
	$$-\int_{-\alpha}^{\alpha} g(\phi) (1-\phi)\phi^\prime u^\prime\leq-\int_{-\alpha}^{\alpha} g(\phi) (1-\phi)\phi^\prime u^\prime_2.$$
	Using Lemma \ref{cotauprima} and considering the sign of \(\phi^\prime\), we find:
		$$-\int_{-\alpha}^{\alpha} g(\phi) (1-\phi)\phi^\prime u^\prime\leq  \|u^\prime_2\|_\infty\int_{-\alpha}^{\alpha}-\phi^\prime\leq\|u^\prime_2\|_\infty\leq\frac{b}{4\Lambda^2}\|\phi^\prime\|_{2}^2,$$
	where we have taken into account that \(g(\phi) \leq 1\). Therefore,  \eqref{igualdadexp} can be written as follows:
	\begin{equation}
		\int_{-\alpha}^{\alpha}\phi^{\prime2} +\int_{-\alpha}^{\alpha} g(\phi)\phi(1-\phi)^2\leq\frac{\sigma}{2}-\phi^\prime(\alpha)+\frac{b}{4\Lambda^2}\int_{-\alpha}^{\alpha}\phi^{\prime2}.
		\label{desigualdad1}
	\end{equation}
	
	Let us now estimate the value of \(\phi'(\alpha)\). This value can be computed, and its expression is:
	$$\phi^\prime(\alpha)=-\theta\frac{\sigma e^{-\sigma \alpha}}{1-e^{-\sigma\alpha}}.$$
	To estimate \(\phi'(\alpha)\), we solve the boundary value problem given by \(\phi'' + \sigma \phi' = 0\), with \(\phi(0) = \theta\) and \(\phi(\alpha) = 0\), noting that \(g(\phi) = 0\) for \(\xi \in [0, \alpha]\). 
From this, we obtain:
\[
|\phi'(\alpha)| \leq \frac{\theta}{\alpha} + |\sigma| \theta.
\]
Substituting this estimate into equation \eqref{desigualdad1}, we get:
	\begin{equation}
		 \left(1-\frac{b}{4\Lambda^2}\right)\int_{-\alpha}^{\alpha}|\phi^\prime|^2+\int_{-\alpha}^{\alpha} g(\phi)\phi(1-\phi)^2\leq\frac{\sigma}{2}+\frac{\theta}{\alpha}+|\sigma|\theta.
		\label{desigualdadHom}
	\end{equation}	

	On the other hand, we have: 
	$$\left(1-\frac{b}{4\Lambda^2}\right)\int_{-\alpha}^{\alpha}|\phi^\prime|^2+\int_{-\alpha}^{\alpha} g(\phi)\phi(1-\phi)^2\geq0. $$
	Therefore, since $\theta\in(0,1/3)$, it follows that
	$$\sigma\geq-\frac{6}{5}\frac{\theta}{\alpha}. $$
	
	Finally, since \(\theta \in (0,1/3)\), and by utilizing Proposition \ref{cotainferior} and Lemma \ref{cotauprima}, we obtain:
	\begin{equation}
		\left(1-\frac{b}{4\Lambda^2}\right)\int_{-\alpha}^{\alpha}|\phi^\prime|^2+\int_{-\alpha}^{\alpha}g(\phi)\phi(1-\phi)^2\leq2+\frac{a}{\Lambda}+\frac{\theta}{\alpha}+\frac{b}{4 \Lambda^2}\int_{-\alpha}^{\alpha}|\phi^\prime|^2.
		\label{desigualdadHom}
	\end{equation}	
	From which we deduce the estimation established in the statement of the Proposition.
\end{proof}

Before proceeding to prove the existence of a solution, let us introduce a preliminary result that will be used to show that \(\phi \in C^1(-\alpha, \alpha)\).

\begin{lemma}
	Let $\phi$ and $u$ be solutions of \eqref{probih}, then  we have
	\begin{equation}
		\|\phi^\prime\|_\infty\leq \frac{3(|\sigma|+\|u^\prime\|_\infty)}{2}+\frac{3\Lambda}{4} +\frac{3}{2\Lambda},
	\end{equation}
	for $\alpha>\frac{\Lambda}{2}\log(3)$.
	\label{cotaderivada}
\end{lemma}
\begin{proof}
	Let $\xi\in[-\alpha,\alpha]$ and consider the expression
	$$\int_{-\alpha}^{\xi} \frac{-e^{-\frac{(\xi-y)}{\Lambda}}}{2}\phi^{\prime\prime}(y)dy + \int_{\xi}^{\alpha} \frac{e^{\frac{(\xi-y)}{\Lambda}}}{2}\phi^{\prime\prime}(y)dy.$$
	Integrating each terms by part, we obtain
	$$\int_{-\alpha}^{\xi} \frac{-e^{-\frac{(\xi-y)}{\Lambda}}}{2\Lambda^2}\phi^{\prime\prime}(y)dy=-\frac{\phi^\prime(\xi)}{2}+\frac{e^{\frac{-(\xi+\alpha)}{\Lambda}}}{2}\phi^\prime(-\alpha)+\int_{-\alpha}^{\xi} \frac{e^{-\frac{(\xi-y)}{\Lambda}}}{2\Lambda}\phi^{\prime}(y)dy,$$
	$$\int_{\xi}^{\alpha} \frac{e^{\frac{(\xi-y)}{\Lambda}}}{2}\phi^{\prime\prime}(y)dy=-\frac{\phi^\prime(\xi)}{2}+\frac{e^{\frac{(\xi-\alpha)}{\Lambda}}}{2}\phi^\prime(\alpha)+\int_{\xi}^{\alpha} \frac{e^{\frac{(\xi-y)}{\Lambda}}}{2\Lambda}\phi^{\prime}(y)dy.$$
	
	Combining both expressions and rearranging terms, we can express \(\phi'\) as follows:
	$$-\phi^\prime(\xi)=I_1+I_2+I_3,$$ 
	where
\begin{align*}
	I_1 &= -\frac{e^{\frac{-(\xi+\alpha)}{\Lambda}}}{2}\phi^\prime(-\alpha) - \frac{e^{\frac{(\xi-\alpha)}{\Lambda}}}{2}\phi^\prime(\alpha) , \\
	I_2 &= -\int_{-\alpha}^{\xi} \frac{e^{-\frac{(\xi-y)}{\Lambda}}}{2\Lambda}\phi^{\prime}(y) \,dy - \int_{\xi}^{\alpha} \frac{e^{\frac{(\xi-y)}{\Lambda}}}{2\Lambda}\phi^{\prime}(y) \,dy , \\
	I_3 &= -\int_{-\alpha}^{\xi} \frac{e^{-\frac{(\xi-y)}{\Lambda}}}{2}\phi^{\prime\prime}(y) \,dy + \int_{\xi}^{\alpha} \frac{e^{\frac{(\xi-y)}{\Lambda}}}{2}\phi^{\prime\prime}(y) \,dy .
\end{align*}

	Let us analyze each term separately. First, we have:
	$$I_1\leq\left(\frac{1}{2}+\frac{e^{\frac{-2\alpha}{\Lambda}}}{2}\right)\|\phi^\prime\|_\infty.$$

	Then, $I_2$ can be bounded as follows
	$$I_2\leq\frac{1}{2\Lambda}\int_{-\alpha}^{\alpha}|\phi^\prime(y)|dy=\frac{1}{2\Lambda},$$
	where we have used the identity
	$$\int_{-\alpha}^{\alpha}|\phi^\prime(y)|dy=\int_{-\alpha}^{\alpha}-\phi^\prime(y)dy=-\phi(\alpha)+\phi(-\alpha)=1,$$
	since $\phi^\prime\leq0$.
	
Finally, since \(\phi\) is a solution to \eqref{probih}, \(I_3\) can be expressed as follows:
	\begin{align*}
		-\int_{-\alpha}^{\xi} \frac{e^{-\frac{(\xi-y)}{\Lambda}}}{2}\left[ -(\sigma+g(\phi)u^\prime(y))\phi^\prime(y)-g(\phi)\phi(1-\phi)(y)\right] dy \\
		+\int_{\xi}^{\alpha} \frac{e^{\frac{(\xi-y)}{\Lambda}}}{2}\left[ -(\sigma+g(\phi)u^\prime(y))\phi^\prime(y)-g(\phi)\phi(1-\phi)(y)\right] dy.
	\end{align*}
	Thus, we find
	$$\int_{-\alpha}^{\xi} \frac{e^{-\frac{(\xi-y)}{\Lambda}}}{2}\phi^{\prime\prime}(y)dy - \int_{\xi}^{\alpha} \frac{e^{\frac{(\xi-y)}{\Lambda}}}{2}\phi^{\prime\prime}(y)dy\leq \frac{(|\sigma|+\|u^\prime\|_\infty)}{2}+\frac{\Lambda}{4} .$$
	Therefore, combining the estimates for \(I_1\), \(I_2\), and \(I_3\), we obtain:
	$$-\phi^\prime(\xi)\leq\left(\frac{1}{2}+\frac{e^{\frac{-2\alpha}{\Lambda}}}{2}\right)\|\phi^\prime\|_\infty+ \frac{(|\sigma|+\|u^\prime\|_\infty)}{2}+\frac{\Lambda}{4} +\frac{1}{2\Lambda}.$$
	Taking norms, we deduce
	$$	\left(\frac{1}{2}-\frac{e^{\frac{-2\alpha}{\Lambda}}}{2}\right)\|\phi^\prime\|_\infty\leq \frac{(|\sigma|+\|u^\prime\|_\infty)}{2}+\frac{\Lambda}{4} +\frac{1}{2\Lambda}.$$
	Then, since \(\alpha > \frac{\Lambda}{2} \log(3)\), we conclude the statement of  Lemma \ref{cotaderivada}.	
\end{proof}

Once we have established the a priori estimates, we can proceed to prove the existence of a solution to \eqref{probih}.

\begin{prop}
	Assume that \(\frac{b}{2\Lambda^2}<1\), \(\theta \in (0, \frac{1}{3})\), and \(\alpha > \max\left\lbrace \alpha_0, \frac{\Lambda}{2}\log(3) \right\rbrace \). Then, there exists a bounded solution of (\ref{probih}) verifying:
	\begin{equation}
	\int_{-\alpha}^{\alpha}|\phi^\prime|^2+\int_{-\alpha}^{\alpha}\tau \phi(1-\phi)^2\leq C_3+C_4\frac{\theta}{\alpha},
	\end{equation}
	for $$-\frac{6}{5}\frac{\theta}{\alpha}\leq\sigma< C_1+C_2\frac{\theta}{\alpha},$$ 
	where \( C_i > 0 \) for \( i \in \left\lbrace 1, 2, 3, 4 \right\rbrace \), are constants independent of \(\theta\) and \(\alpha\).	
	\label{homotopia}	
\end{prop}

\begin{proof}
	Let us consider the application $$F_\tau:(\sigma,\phi,u)\rightarrow(\theta_\tau,\Phi_\tau,U_\tau)$$
	where $\Phi_\tau$ is the solution of the differential equation:
	\begin{equation}
	\begin{array}{l}
	{-\Phi_\tau^{\prime\prime}-\sigma\Phi_\tau^\prime-g(\phi)\tau u^\prime\Phi_\tau^\prime=\tau g(\phi)\phi(1-\phi)},\\
	{\Phi_\tau(-\alpha)=1, \quad \Phi_\tau(+\alpha)=0.}
	\end{array}
	\label{probih2}
	\end{equation} 
	The function \( U_\tau=\tau\Gamma* \left( a \bar{\phi}+\frac{b}{2}(\bar{\phi}^\prime)^2 \right) \), where \(\bar{\phi}\) is the extension by constant to all \(\mathbb{R}\) of \(\phi\), satisfies 
	$${-\Lambda^2 U_\tau^{\prime\prime}+U_\tau=\tau a\bar{\phi}+\frac{b}{2}(\bar{\phi}^\prime)^2}.$$
	The value $\theta_\tau$ is given by the expression:
	$$\theta_\tau=\theta-\max_{\xi\geq0}\phi(\xi)+\sigma.$$

The operator \( F_\tau \) maps the Banach space \( \Omega = \mathbb{R} \times C^{1}([-\alpha,\alpha]) \times C^{1}([-\alpha,\alpha]) \) onto itself
$$\|(\sigma,\phi,U)\|_{\Omega}=\max\left\lbrace |\sigma|,\|\phi\|_{ C^{1}([-\alpha,\alpha])},\|U\|_{ C^{1}([-\alpha,\alpha])}\right\rbrace .
$$.
	
	Let us consider the ball \( B_M = \left\lbrace (\sigma, \phi, U) \in \Omega \;|\; \|(\sigma, \phi, U)\|_{\Omega} < M \right\rbrace \). We can find a sufficiently large \( M \) such that the operator \( I - F_\tau \) does not cancel on the boundary of \( B_M \) for all \( \tau \in [0, 1] \). This is equivalent to obtaining a uniform bound for \( |\sigma| \), \( \|\phi\|_{C^1} \), and \( \|U\|_{C^1} \) for any solution of \eqref{probih}.

Let us analyze the uniform bounds. By Lemma \ref{propiedadessol}, we have that \( 0 \leq \phi \leq 1 \). Using Proposition \ref{propinf}, we obtain that \( \phi^\prime \) is bounded in \( L^2(-\alpha,\alpha) \), which implies, thanks to Lemma \ref{cotauprima}, that \( U \in C^1([-\alpha,\alpha]) \).

Furthermore, by Lemma \ref{cotainferior} and Proposition \ref{propinf}, we have uniform bounds for \( \sigma \). Lemma \ref{cotaderivada} assures that \( \phi^\prime \) is bounded in \( C([-\alpha,\alpha]) \). This leads us to find the bound \( M \).

In addition, \( F_\tau \) is absolutely continuous and depends continuously on the parameter \( \tau \), due to the \( C^1([-\alpha,\alpha]) \) bounds obtained.

Therefore, we are able to estimate the Leray-Schauder degree  \( \deg(I - F_\tau, B_M, 0) \) (see for instance \cite{BerestyckiLions,Fonda,Mawhin}), which is well defined and independent of \( \tau \). If we prove \( \deg(I - F_0, B_M, 0) \neq 0 \), then, due to homotopy invariance, \( \deg(I - F_1, B_M, 0) \neq 0 \). By the degree property, there exists a fixed point of \( F_1 \) in \( B_M \).
	
	Let us calculate \(\deg(I - F_0, B_M, 0)\). For \(\tau = 0\), the operator \(F_0\) simplifies to:
	$$\Phi_{0,\sigma}(\xi)=\frac{e^{-\sigma \xi}-e^{-\sigma \alpha}}{e^{\sigma \alpha}-e^{-\sigma\alpha}}.$$
	Therefore, $F_0$ has the following expression
	$${F}_0(\sigma,\phi,U)=\left(\sigma+\theta-\frac{1-e^{-\sigma \alpha}}{e^{\sigma \alpha}-e^{-\sigma\alpha}},\frac{e^{-\sigma \xi}-e^{-\sigma \alpha}}{e^{\sigma \alpha}-e^{-\sigma\alpha}},0\right).$$
	Consequently, we have
	$$(I-{F}_0)(\sigma,\phi,U)=\left(\frac{1-e^{-\sigma \alpha}}{e^{\sigma \alpha}-e^{-\sigma\alpha}}-\theta,\phi-\frac{e^{-\sigma\xi}-e^{-\sigma \alpha}}{e^{\sigma \alpha}-e^{-\sigma\alpha}},U\right).$$
	One can find a \(\bar{\sigma}\) that satisfies:
	\begin{equation}
	\frac{1-e^{-\bar{\sigma} \alpha}}{e^{\bar{\sigma} \alpha}-e^{-\bar{\sigma}\alpha}}=\theta.
	\label{homosigma}
	\end{equation}
	
	The degree of ${F}_0$ is equivalent to the degree of the function:
	$$\mathcal{H}(\sigma,\phi,U)=\left(\theta_0-\frac{1-e^{-\sigma \alpha}}{e^{\sigma \alpha}-e^{-\sigma\alpha}},\phi-\frac{e^{-\bar{\sigma} x}-e^{-\bar{\sigma} \alpha}}{e^{\bar{\sigma} \alpha}-e^{-\bar{\sigma}\alpha}},U\right):=(H(\sigma),I_\phi-\Phi_{0,\bar{\sigma}},I_U).$$

	Therefore, \(\deg(I - F_0, B_M, 0) = \deg(H(\sigma), J, 0)\), where \(J\) is an open set containing the unique root of \eqref{homosigma}, by the product of degrees property. This yields \(\deg(I - F_0, B_M, 0) = -1\), thereby proving the existence of a fixed point for the operator \(F_\tau\).
\end{proof}

Once we have established the existence of a solution in \([- \alpha, \alpha]\), we can take the limit as \(\alpha \to \infty\).

\begin{prop}
	If $\frac{b}{2\Lambda^2}<1$, and $\theta\in(0,\frac{1}{3})$, then there exists a bounded solution to
		\begin{equation}
	\begin{array}{l}
	{-\sigma\phi^\prime-\phi^{\prime\prime}-g(\phi)u^\prime\phi^\prime= g(\phi)\phi(1-\phi)},\\
	{\phi(-\infty)=1, \quad \phi(+\infty)=0,\quad \phi(0)=\theta},
	\end{array}
	\label{prob12}
	\end{equation} 
where $u$ is defined as $u=\Gamma* \left(a\phi+\frac{b}{2}\phi^{\prime2}\right)$, and $\phi$ satisfies
	\begin{equation}
		\int_{-\alpha}^{\alpha}\phi^{\prime2}+\int_{-\alpha}^{\alpha} g(\phi)\phi(1-\phi)^2\leq C_3,
	\end{equation}
	for $0\leq\sigma< C_1$, being $C_1$ and $C_3$ constants defined in  Proposition \ref{homotopia}.
	\label{alphainf}
\end{prop}

\begin{proof}
	Consider an increasing sequence of intervals \(\{[-\alpha_n, \alpha_n]\}_n\) with \(\alpha_n \to \infty\), and a corresponding sequence \(\{\sigma_n, \phi_n, U_n\}_n\), where \((\sigma_n, \phi_n, U_n)\) is the solution to \eqref{probih} on each interval \([- \alpha_n, \alpha_n]\). These solutions exist due to Proposition \ref{homotopia}. Moreover, the uniform bounds established for \(\sigma_n\), \(\phi_n\), \(u_n\), and their derivatives, which are independent of \(\alpha_n\), ensure that \(\sigma_n\) converges to \(\hat{\sigma}\), up to a subsequence.
Due to these uniform bounds, \(\phi_n \to \hat{\phi}\) and \(u_n \to \hat{u}\), up to a subsequence, in the sense of uniform convergence on compact subsets and their derivatives. Consequently, \(\hat{\sigma}\), \(\hat{\phi}\), and \(\hat{u}\) satisfy the differential equation:
	\begin{equation}
	\begin{array}{l}
	{-\hat{\sigma}\hat{\phi}^\prime-\hat{\phi}^{\prime\prime}-g(\hat{\phi})\hat{u}^\prime\hat{\phi}^\prime= g(\hat{\phi})\hat{\phi}(1-\hat{\phi})},\\
	{-\Lambda^2 \hat{u}^{\prime\prime}+\hat{U}= a\hat{\phi}+\frac{b}{2}\hat{\phi}^{\prime2}}.
	\end{array}
	\end{equation} 
	
	It only remains to prove that $\hat{\phi}(+\infty)=0$ and $\hat{\phi}(-\infty)=1$.
	
	Since $\|\bar{\phi}^\prime\|_{2}$ is bounded, there exist the values $\hat{\phi}(+\infty)$ and $\hat{\phi}(-\infty)$. By Proposition \ref{homotopia} we have:
	\begin{equation}
		\int_{-\infty}^{\infty} g(\phi)\bar{\phi}(1-\bar{\phi})^2<+\infty.
		\label{intlim}
	\end{equation}
	
	Let us see that $\hat{\phi}(+\infty)=0$. Note that $\phi_n(\xi)=\theta\frac{e^{-\sigma \xi}-e^{-\sigma \alpha}}{e^{\sigma \alpha}-e^{-\sigma\alpha}}$, for $\xi\geq0$. Then, we have  $\hat{\phi}(\xi)=\theta\frac{e^{-\hat{\sigma} \xi}-e^{-\hat{\sigma} \alpha}}{e^{\hat{\sigma} \alpha}-e^{-\hat{\sigma}\alpha}}$. Therefore, $\bar{\phi}(+\infty)=0$.
	
	On the other hand, the monotonicity of \(\phi_n\) is inherited by \(\hat{\phi}\). Since \(\phi_n^\prime(0) < 0\), it follows that \(\hat{\phi}(-\infty) > \theta\) and \(\hat{\phi}(\infty) = 1\), due to the boundedness of the integral in \eqref{intlim}.
\end{proof}

Finally, to complete the proof of the existence of a solution, it remains to let \(\theta \rightarrow 0\).

\begin{prop}
	If $\frac{b}{2\Lambda^2}<1$, then there exist a bounded solution to
	\begin{equation}
		\begin{array}{l}
			{-\sigma\phi^\prime-\phi^{\prime\prime}-u^\prime\phi^\prime= \phi(1-\phi)},\\
			{\phi(-\infty)=1, \quad \phi(+\infty)=0,}
		\end{array}
		\label{prob13}
	\end{equation} 
where $u$ is defined as $u=\Gamma* \left(a\phi+\frac{b}{2}\phi^{\prime2}\right)$, and $\phi$ satisfies
	\begin{equation}
		\int_{-\alpha}^{\alpha}|\phi^\prime|^2+\int_{-\alpha}^{\alpha} \phi(1-\phi)^2\leq C_3,
	\end{equation}
	for $0\leq\sigma< C_1$, where $C_1$ and $C_3$ are constants defined in Proposition \ref{homotopia}.
	\label{prop13}
\end{prop}
\begin{proof}
	Let us consider a decreasing sequence \(\theta_n \rightarrow 0\) and the associated solutions \((\sigma_n, \phi_n, U_n)\) to \eqref{prob12}. Choose a translation \(\xi\) such that \(\phi_n(0) = \frac{1}{2}\) for all \(n \geq 1\). As \(\theta \rightarrow 0\), and since \(g(\phi) \rightarrow 1\), the convergence properties follow a proof scheme similar to that used in Proposition \ref{alphainf}. 
Thus, we obtain that \(\sigma_n \rightarrow \tilde{\sigma}\), \(\phi_n \rightarrow \tilde{\phi}\), and \(u_n \rightarrow \tilde{u}\) up to a subsequence, with uniform convergence on compact sets. Additionally, \(\phi^\prime_n \rightarrow \tilde{\phi}^\prime\) and \(u^\prime_n \rightarrow \tilde{u}^\prime\) uniformly on compact sets. Consequently, \(\tilde{\sigma}\), \(\tilde{\phi}\), and \(\tilde{u}\) satisfy the differential equation:
	\begin{equation}
		\begin{array}{l}
			{-\sigma\phi^\prime-\phi^{\prime\prime}-u^\prime\phi^\prime= \phi(1-\phi)},\\
			{-\Lambda^2 u^{\prime\prime}+u= a\phi+\frac{b}{2}(\phi^\prime)^2},
		\end{array}
	\end{equation} 
	Finally, analogous to the proof of Proposition \ref{alphainf}, it can be shown that \(\tilde{\phi}(-\infty) = 1\) and \(\tilde{\phi}(+\infty) = 0\).
\end{proof}

	Theorem \ref{Teorema3} guarantees the existence of traveling waves, and from the estimates derived in the various results, it can be shown that \(\sigma\) lies within the interval:
\[ \sigma \in \left[0, 2 + \frac{a}{\Lambda} + \frac{b}{4\Lambda^2} \frac{2 + \frac{a}{\Lambda}}{1 - \frac{b}{2\Lambda^2}}\right]. \]

To demonstrate that the obtained \(\sigma\) is always greater than \(2\), we use the following reasoning:

Given the bounds on \(u'\) provided by the estimates on \(\phi\) and \(\phi'\), and the expression \(u = \Gamma * \left(a\phi + \frac{b}{2} (\phi')^2\right)\), we know that \(u'\) is bounded. Moreover, it can be shown that \(\phi'(\pm \infty) = 0\), \(u'(\pm \infty) = 0\), and \(u(-\infty) = a\), by repeatedly applying this generalized Lasalle's Invariance Theorem.
	\begin{lemma}{(Lasalle's Invariance Theorem, revisited)}
		Let \( f: \mathbb{R}^N \to \mathbb{R}^N \) be a continuous function, and consider the differential equation \( x' = f(x) \). Let \( V: \mathbb{R}^N \to \mathbb{R} \) be a \( C^1 \) function, and suppose \( x: [\alpha, +\infty) \to \mathbb{R}^N \) is a positively bounded solution such that \( V(x(t)) \to L \), as \( t \to +\infty \). Then we have:
\[
\frac{d}{dt} V(x(t)) \to 0.
\]
	\end{lemma}
	\begin{proof}
		Let \( t_n \) be a sequence such that \( x(t_n) \to x_0 \). Consider a closed ball \(\bar{B}(x_0, r)\) and let \( K = \max_{x \in \bar{B}(x_0, r)} |f(x)| \). Choose \( n_0 \) such that \( x(t_n) \in \bar{B}(x_0, r/2) \) for all \( n \geq n_0 \). Select \( \epsilon \leq \frac{r}{2K} \) such that \( x(t) \in B(x_0, r) \) for all \( t \in [t_n - \epsilon, t_n + \epsilon] \).

Define \( z_n: [-\epsilon, \epsilon] \to \mathbb{R}^N \) by
\[ z_n(t) = x(t + t_n). \]
This function \( z_n \) represents a uniformly bounded and equicontinuous family of solutions. By the Ascoli-Arzel\`a theorem, there exists a subsequence \( \{ z_{n_k} \} \) such that
\[ z_{n_k} \to z_0 \]
uniformly, where \( z_0: [-\epsilon, \epsilon] \to \mathbb{R}^N \) is a solution of the differential equation with \( z_0(0) = x_0 \).

Next, let us prove that \( V(z(t)) \) is constant. We have:
\[ V(z(t)) = \lim_{n \to \infty} V(z_n(t)) = \lim_{n \to \infty} V(x(t + t_n)) = \lim_{t \to \infty} V(x(t)) = L. \]
Thus, \( V(z(t)) \) is constant, and therefore
\[ \frac{d}{dt} V(z(t)) = 0. \]
In particular, at \( t = 0 \),
\[ \frac{d}{dt} V(z(0)) = 0 = \frac{d}{dt} V(x_0) = \lim_{n \to \infty} \frac{d}{dt} V(x(t_n)). \]
This completes the proof.
	\end{proof}
	
Given \( V(\phi, \phi', u, u') = \phi \), we deduce that \(\phi'(\pm \infty) = 0\). Using the expression for \( u \), we get that \( u(-\infty) = a \) and \( u(+\infty) = 0 \). Applying \( V(\phi, \phi', u, u') = u \), we also find that \( u'(\pm \infty) = 0 \).

Now, we can proceed to show that there are no monotonically decreasing traveling wave solutions for \(\sigma < 2\).	\begin{lemma}
		Given that \(\sigma\) and \(\phi\) are bounded solutions of \eqref{prob13}, it follows that \(\sigma \geq 2\).
	\end{lemma}
	\begin{proof}
		To prove that \(\sigma \geq 2\) by contradiction, assume that \(\sigma < 2\). We know that if \(\phi\) is a solution of \eqref{prob13}, it is monotonically decreasing, satisfies \(\phi(+\infty) = 0\), and \(\phi^\prime(+\infty) = 0\). Additionally, \(u^\prime(+\infty) = 0\), so there exists a sufficiently large \(\bar{\xi}\) such that \(\|u^\prime(\xi)\| < \epsilon\) for \(\xi > \bar{\xi}\).

Consider the differential equation \eqref{prob13} and make the following change of variables to polar coordinates:
\[
\begin{split}
\phi(\xi) &= r(\xi) \cos(\omega(\xi)),
\\
\phi^\prime(\xi) &= r(\xi) \sin(\omega(\xi)).
\end{split}
\]
By substituting these into the differential equation, we obtain the transformed differential equation in terms of \(r\) and \(\omega\):
		\[
		\begin{split}
			r^\prime&=-(\sigma+u^\prime)r\sin^2(\omega)+r^2\cos^2(w)\sin(w),\\
			\omega^\prime&=-1-(\sigma+u^\prime)\frac{\sin(2\omega)}{2}+r\cos^3(\omega).
		\end{split}
		\]
		
		If we demonstrate that the solution is not monotonically decreasing but instead exhibits infinite oscillations before reaching zero -equivalent to saying that \(\omega(\xi)\) remains consistently positive or negative for large values of \(\xi\)- we will have reached a contradiction.
		
		We can bound \(\omega'\) as follows:
		$$ \omega^\prime=-1-(\sigma+u^\prime)\frac{\sin(2\omega)}{2}+r\cos^3(\omega)\leq -1+\frac{1}{2}(\sigma+u^\prime)+r.$$		
		Since \( r \rightarrow 0 \) as \( \xi \rightarrow +\infty \), \(\omega'\) does not change sign when \(\frac{\sigma + u'}{2} + 1 < 0\). This leads to a contradiction, as it implies that the solution would spiral toward zero.
	\end{proof}

\bibliography{bibliography}
\bibliographystyle{acm}

\end{document}